  \documentclass[twoside,reqno,11pt]{fcaa-var} %

\usepackage{graphicx}
\usepackage{epsfig}
\usepackage{amsthm}
\usepackage{amsmath}
\usepackage{latexsym}
\usepackage{amsfonts}
\usepackage{amssymb}

\newtheoremstyle{theorem}
  {15pt}          
  {15pt}  
  {\sl}  
  {\parindent}
  {\sc}  
  {. }   
  { }    
  {}     
\theoremstyle{theorem}

\newtheorem{theorem}{Theorem}[section]
\newtheorem{corollary}{Corollary}[section]

\newtheoremstyle{defi}
  {15pt}          
  {15pt}  
  {\rm}  
  {\parindent}     
  {\sc}  
  {. }    
  { }    
  {}     
\theoremstyle{defi}

\newtheorem{remark}{Remark}[section]

 \def\proofname{\indent\rm P\ r\ o\ o\ f.}
 \def\proofend{\hfill$\Box$}
 \def\qed{\hfill$\Box$} 

 \usepackage{hyperref}

 \def\theequation{\arabic{section}.\arabic{equation}}

 \def\themycopyrightfootnote{\vspace*{3pt}
 \copyright \, Year\,  Diogenes  Co., Sofia
   \par  \noindent pp. xxx--xxx, DOI: ......................
   \hfill  \vspace*{-36pt}
  }

 \title[Robustness and convergence of fractional \dots]
       {Robustness and convergence of fractional \\ [3pt] systems and
        their 
        applications to \\ [3pt]  adaptive schemes  }
 \author[\normalsize J. A. Gallegos, M. A. Duarte-Mermoud]
        {\normalsize Javier A. Gallegos $^1$, Manuel A. Duarte-Mermoud $^2$}

\begin{document}

\noindent
Submitted to: \emph{Fract. Calc. Appl. Anal.} (http://www.degruyter.com/view/j/fca), Vol. 20 (2017).

 \bigskip \medskip

 \begin{abstract}

Our general aim is to give sufficient conditions for robustness behavior and convergence to the equilibrium point of linear time-varying fractional system's solutions. We approach this problem using as a framework a series of recent results due to Cong et al. We establish theorems that generalize in several ways many previous results in the specialized literature, including those of Cong et al. We use the proposed theorems in adaptive systems, proving convergence and robustness of such schemes, that up to date remain unsolved problems, showing the wide scope of application of our results.

 \medskip

{\it MSC 2010\/}: Primary 26A33; Secondary 34A08, 34K37, 62F35

 \smallskip

{\it Key Words and Phrases}: robustness, fractional differential equations, adaptive systems

 \end{abstract}

 \maketitle

 \vspace*{-16pt}



 \section{Introduction}\label{sec:1}

\setcounter{section}{1}
\setcounter{equation}{0}\setcounter{theorem}{0}

In a recent series of papers [1-6], Cong et al. have presented a powerful fixed point technique to asymptotically analyze Caputo fractional systems with $\alpha \in (0,1)$ described by the following equation,
\begin{equation} \label{Linear1}
\begin{cases} D^{\alpha} x(t) = (A + Q(t))x(t)  \\ x(0) = x_0 \end{cases} ,
\end{equation}
where we have implicitly assumed that the initial time for the derivative $t=0$, is the same that the time for the initial condition,  with $x \in \mathbb{R}^n$ and $x_0 \in\mathbb{R}^n $ is an arbitrary initial condition. They were able to prove (\cite[Theorem 5]{C.Asym}) that if $Q : [0, \infty) \rightarrow \mathbb{R}^{n \times n}$ is a continuous matrix function such that $q = \sup_{t \geq 0} \int_0 ^t \tau^{\alpha-1}  || E_{\alpha, \alpha}(A,\tau)  [Q(t - \tau)]|| d\tau < 1$ and if the spectrum $\sigma(A)$ of $A \in \mathbb{R}^{n \times n}$ is such that $\sigma(A) \subset \{ \lambda \in \mathbb{C} \setminus \{0\}: |\arg (\lambda)| > \frac{\alpha \pi}{2} \}$,  then $x$ is bounded and $\lim_{t\rightarrow \infty} x(t)=0$.   Actually, they show that instead of condition $q<1$  it is enough that  $\sup_{t\geq 0} ||Q(t)||$ be small enough or $\lim_{t\rightarrow \infty} ||Q(t)|| = 0$ (\cite[Theorems 7, 8]{C.Asym}).

Under the same assumption for $A$ and similar arguments, they prove in \cite[Theorem 5]{C.Lin} that for the system
\begin{equation} \label{Lin}
\begin{cases} D^{\alpha} x(t) = Ax(t) + f(x(t)) \\ x(0) = x_0 \end{cases} ,
\end{equation}
where $f$ is locally Lipschitz around the origin with $f(0)=0$ with Lipschitz constant $L(r)$ such that $\lim_{r \rightarrow 0^+} L(r) =0$ then $\lim_{t \rightarrow \infty} ||x(t)|| = 0$ for any $x_0$ close enough to the origin and $||x(0)|| < r$. Roughly speaking, one can always find a sufficiently small $r$ such that $\sup_{t} ||Q(t)=L(r)I||$ is sufficiently small, and therefore, to apply the previous result. Also, for a system (\ref{Lin}) where $A$ is such that $\sigma(A) \subset \{ \lambda \in \mathbb{C} \setminus \{0\}: |\arg (\lambda)| < \frac{\alpha \pi}{2} \}$, they prove that the origin $x=0$ is unstable  (\cite[Theorem 5]{C.Instability}). Hence, The first Lyapunov Theorem on stability and instability for fractional systems follows as a corollary, i.e. for an autonomous system $D^{\alpha} x = f(x(t))$ the local asymptotic stability is related to the eigenvalues of the Jacobian of $f$, a more general result that the one obtained using Lyapunov approach in \cite[Example 4]{FLT}. Next, they provide a Lipschitz characterization of the stable manifold for system (\ref{Lin}) where $A$ can have eigenvalues belonging to $\{ \lambda \in \mathbb{C} \setminus \{0\}: |\arg (\lambda)| < \frac{\alpha \pi}{2} \}$  (\cite[Theorem 4]{C.Manifold}).

For the following system
\begin{equation} \label{Linear2}
\begin{cases} D^{\alpha} x(t) = Ax(t) + f(t)  \\ x(0) = x_0 \end{cases}
\end{equation}
with $f$ a continuous function, they show a distinctive feature of fractional systems, namely that there exists a solution (there exists an initial condition) so that (i) if $f$ is bounded, $x$ is bounded (ii) if $f$ converges to zero, $x$ converges to zero  if and only if $\sigma(A) \subset \{ \lambda \in \mathbb{C} \setminus \{0\}: |\arg (\lambda)| > \frac{\alpha \pi}{2} \} \cup  \{ \lambda \in \mathbb{C} \setminus \{0\}: |\arg (\lambda)| < \frac{\alpha \pi}{2} \} $ (at least one eigenvalue of A must be unstable) (\cite[Theorem 13]{C.Perron}).

In a somewhat different matter, though fractional systems are not dynamical ones (they show it with a intersection of trajectories argument; by other means, this was proved in \cite[Proposition 2]{FLT}), some fractional systems define a non local dynamical system, i.e. it is possible to construct a 2-flow $\varphi_{s,t}$ such that $\varphi_{s,t} \> o \> \varphi_{u,s} = \varphi_{u,t}$. It is achieved by choosing $\varphi_{s,t}= \varphi_{0,t} \> o \> \varphi_{0,s}^{-1}$ where $\varphi_{0,t} (x_0)$ is interpreted as $x(t)$ given that the initial condition in $x(0)=x_0$ (\cite[Theorem 13]{C.Flow}). However, the interpretation of $\varphi_{s,t}$ cannot be $x(t)$ given the initial condition in $s>0$, as in dynamic flows. In fact, consider a particular linear system $\varphi_{0,t}(x_0)= x_0 E_{\alpha }(t^{\alpha})$. Then $\varphi_{0,s}^{-1} (x(s))= x_0$ where obviously $x(s)= x_0 E_{\alpha }(s^{\alpha})$. But $\varphi_{s,t} (x(s)) = x(s) E_{\alpha}((t-s)^{\alpha})= \varphi_{s,t} \> \> o \> \> \varphi_{0,s} (x_0)=  x_0 E_{\alpha }(s^{\alpha}) E_{\alpha}((t-s)^{\alpha}) \neq x(0)E_{\alpha }(t^{\alpha})=\varphi_{0,t} (x_0)$ since $E_{\alpha }(t^{\alpha}) \neq E_{\alpha }(s^{\alpha}) E_{\alpha}((t-s)^{\alpha})$ provided that $\alpha \neq 1$.

We remark that the alternative approach to study convergence using Lyapunov functions (\cite{FLT, FS}) has the following disadvantage: by defining those functions as scalar ones, positive or negative definite restrictions on the matrix result to establish the sign of its fractional derivatives (\cite[Example 6]{FLT}).

Our contribution is first to extend above theorems to include several new aspects such as: a  wider notion of robustness, the use of time varying linear system, a piece continuity requirement, the case of $\alpha \geq 1$ and other types of derivatives, and the extension for mixed order systems. Secondly, these extensions are to be used in an application of fractional calculus in adaptive control of integer systems. These extensions will guarantee robust behavior of the overall controlled system without modifying the adaptive scheme but requiring conditions on the information signals. On the other hand, the extensions ample the class of information signals that assure error convergence to zero when the perturbation is absent. Therefore, they also extend recent results on this subject in the adaptive control literature. (\cite{FA}).

The paper is organized as follows: Section 2 presents the main contribution of the paper, which is the generalization of the theorems just mentioned. Section 3 shows the usefulness of these generalizations by applying them to determine the convergence and robustness of adaptive scheme designs with fractional operators.

\section{Robustness and convergence}

\setcounter{section}{2}
\setcounter{equation}{0}\setcounter{theorem}{0}

We present theorems that extend in many aspects previous results in the revised literature. Though the arguments  mainly follow the fixed point technique of Cong et al, fine details are required at each argumentation. The results have mathematical interest in its self and will show its  importance in applications in Section 3.

Along this work, we use $||\cdot||$ for the standard norm on $\mathbb{R}^n$ and also for any matrix norm compatible with the standard norm on $\mathbb{R}^n$.

\subsection{Robustness of linear time invariant fractional systems}
\textbf{    }

\textbf{    }

Though \cite[Theorem 5]{C.Asym} is presented as a robustness result, because the convergence property is preserved when $Q(t)$ variations occur around $A$, a more general robustness condition can be stated where just the boundedness property of solutions is preserved. The following result considers this view and, at the same time, is a generalization to linear time varying systems of the sufficient part of \cite[Theorem 13]{C.Perron}.

\begin{theorem}
Consider the following Caputo system
\begin{equation}
D^{\alpha} x(t) = (A + Q(t))x(t) + \nu(t)
\end{equation}
with $\alpha \in (0,1)$, $x(t) \in \mathbb{R}^n$ for all $t \geq 0$ and $\nu : [0, \infty) \rightarrow \mathbb{R}^n$  a bounded continuous function ($ ||\nu||_\infty < \nu_0$). Let $x_0 \in\mathbb{R}^n $ be an arbitrary initial condition at the initial time of the Caputo derivative. Suppose that $Q : [0, \infty) \rightarrow \mathbb{R}^{n \times n}$ is a continuous matrix function such that
\begin{equation}\label{Q}
q := \sup_{t\geq 0} || \int_0 ^t \tau^{\alpha-1}  E_{\alpha, \alpha}(A,\tau)  Q(t - \tau) d\tau || < 1
\end{equation}
for any matrix norm compatible with the standard norm on $\mathbb{R}^n$ and that $A \in \mathbb{R}^{n \times n}$ is a constant matrix such that
\begin{equation} \label{QE}
\sigma(A) \subset \Lambda^s_\alpha := \{ \lambda \in \mathbb{C} \setminus \{0\}: |\arg (\lambda)| > \frac{\alpha \pi}{2} \}
\end{equation}

Then,

(i) $x$ is a bounded continuous function.

(ii) If in addition $ \lim_{t \rightarrow \infty} || \int_0 ^t \tau^{\alpha-1}  E_{\alpha, \alpha}(A,\tau)  \nu (t-\tau) d\tau || = 0$ then $\lim_{t \rightarrow \infty} ||x(t)||=0$. In particular, if $\nu$ converges to zero, then $\lim_{t \rightarrow \infty} ||x(t)||=0$.
\end{theorem}

\proof

Define the Lyapunov-Perron like operator as
\begin{equation*}
\begin{aligned}
\mathcal{T}_{x_0} (\xi) (t) := & E_{\alpha} (t^{\alpha} A) x_0 + \int_0 ^t \tau^{\alpha-1}  E_{\alpha, \alpha}(A,\tau)  Q (t-\tau) \xi(t-\tau) d\tau \\ & + \int_0 ^t \tau^{\alpha-1}  E_{\alpha, \alpha}(A,\tau)  \nu (t-\tau) d\tau.
\end{aligned}
\end{equation*}

Let $\epsilon > 0$ arbitrary and $\xi  \in B_ {{\infty}} (0, \epsilon) := \{ \xi \in \mathcal{C}([0, \infty); \mathbb{R}^n) : ||\xi||_{\infty} \leq \epsilon \} $. Using that $ E_{\alpha} (t^{\alpha} A) x_0$ is bounded and $\int_0 ^t \tau^{\alpha-1} ||E_{\alpha, \alpha}(\tau ^{\alpha} A)|| d\tau < \infty$ since $\sigma(A) \subset \Lambda^s_\alpha$ (\cite[Theorem 3]{C.Asym}), and that $\nu$ is bounded, we have
\begin{equation*}
||\mathcal{T}_{x_0} (\xi) ||_\infty \leq C(\nu_0) < \infty
\end{equation*}
where $C(\nu_0)$ is a constant term depending on the bound $\nu_0$ of function $\nu$. Hence, $\mathcal{T}_{x_0} (B_ {\infty} (0, \epsilon)) \subset B_ {\mathcal{C}_{\infty}} (0, C(\nu_0))$.

Since $q = \sup_{t\geq 0} || \int_0 ^t \tau^{\alpha-1}  E_{\alpha, \alpha}(A,\tau)  [Q(t - \tau)]||$, we have
\begin{equation*}
||\mathcal{T}_{x_0} (\xi) - \mathcal{T}_{x_0} (\tilde{\xi}) ||_\infty \leq q ||\xi - \tilde{\xi}||_\infty.
\end{equation*}

Therefore, $\mathcal{T}_{x_0} $ restricted to the set $\mathcal{C}_{\infty} :=  \{ \xi \in \mathcal{C}([0, \infty); \mathbb{R}^n) : ||\xi||_{\infty} < \infty \} $ is a contraction auto map in a Banach space, whereby it has  a unique fixed point by the contraction mapping theorem in the space $\mathcal{C}_{\infty} $. Since the unique solution in $\mathcal{C}([0, \infty))$ to  $D^{\alpha} x = (A + Q)x + \nu$ can be written as $x(t)= E_{\alpha} (t^{\alpha} A) x_0 + \int_0 ^t \tau^{\alpha-1}  E_{\alpha, \alpha}(A,\tau)   Q (t-\tau) x(t-\tau) d\tau + \int_0 ^t \tau^{\alpha-1}  E_{\alpha, \alpha}(A,\tau)  \nu (t-\tau) d\tau= \mathcal{T}_{x_0}(x)(t)$ (see e.g \cite[\S 3]{Podlubny}), we conclude that $x=\xi$, and thus, $x$ is a bounded continuous function since $\xi \in \mathcal{C}_{\infty}$.

(ii) We have
\begin{equation*}
\begin{aligned}
x(t) = & E_{\alpha}(A,t) x(0) + \int_0 ^t \tau^{\alpha-1}  E_{\alpha, \alpha}(A,\tau)  Q(t-\tau) x(t-\tau) d\tau \\
       & + \int_0 ^t \tau^{\alpha-1}  E_{\alpha, \alpha}(A,\tau)  \nu(t-\tau) d\tau
\end{aligned}
\end{equation*}
and taking standard norm on $\mathbb{R}^n$
\begin{equation*}
\begin{aligned}
||x(t)|| \leq & ||E_{\alpha}(A,t) x(0)|| + ||\int_0 ^t \tau^{\alpha-1}  E_{\alpha, \alpha}(A,\tau)  Q(t-\tau) x(t-\tau) ) d\tau|| \\
     & + ||\int_0 ^t \tau^{\alpha-1}  E_{\alpha, \alpha}(A,\tau)  \nu(t-\tau) d\tau||.
\end{aligned}
\end{equation*}

Define $\epsilon := \frac{ 1- q}{2q +1} \limsup_{t} ||x(t)||$ (by the part (i), we have that $\epsilon < \infty$) and $T$ such that $\sup_{t \geq T} |x(t)| \leq \limsup_{t \rightarrow \infty} ||x(t)|| + \epsilon $ (by the part (i), such a $T$ always exists since $x$ is bounded). Then,
\begin{equation*}
\limsup_{t\rightarrow \infty} ||x(t)|| \leq (\limsup_{t\rightarrow \infty}||x(t)|| + \epsilon) \sup_{t \geq T} ||\int_T ^t \tau^{\alpha-1}  E_{\alpha, \alpha}(A,\tau)  Q (t-\tau) d\tau||
\end{equation*}
where  the hypothesis $\lim_{t\rightarrow \infty} \int_0 ^t \tau^{\alpha-1}  ||E_{\alpha, \alpha}(A,\tau)  \nu (t-\tau)|| d\tau =0$ and the fact that $\lim_{t\rightarrow \infty} ||E_{\alpha}(A,t) ||=0$ since $\sigma(A) \subset \Lambda^s_\alpha$ .

Thus,
\begin{equation*}
\limsup_{t \rightarrow \infty} ||x(t)|| \leq  \limsup_{t \rightarrow \infty} ||x(t)|| \frac {2q + q^2}{2q +1} < \limsup_{t \rightarrow \infty} ||x(t)||
\end{equation*}
which is a contradiction if $\limsup_{t \rightarrow \infty} ||x(t)|| \neq 0$. Hence $\limsup_{t \rightarrow \infty} ||x(t)||=0$, whereby $\lim_{t \rightarrow \infty} ||x(t)|| = 0$.

To conclude, note that
\begin{equation*}
0 \leq ||\int_0 ^t \tau^{\alpha-1}  E_{\alpha, \alpha}(A,\tau)  \nu(t-\tau) d\tau|| \leq \int_0 ^t \tau^{\alpha-1}  || E_{\alpha, \alpha}(A,\tau)|| || \nu(t-\tau)|| d\tau.
\end{equation*}

Since $\sup_{t\geq 0}  \int_0 ^t  \tau^{\alpha-1} ||  E_{\alpha, \alpha}(A,\tau) || d\tau  = C(\alpha, A) < \infty$ (\cite[Theorem 3]{C.Asym}), we have $\tau^{\alpha-1} || E_{\alpha, \alpha}(A,\tau)|| \in \mathcal{L}^1$, the Lebesgue space of integrable functions on $\mathbb{R}_{\geq0}$. Since $\nu$ is bounded and converges to zero, we can apply \cite[Lemma 2]{FS}, to conclude that $\lim_{t\rightarrow \infty} \int_0 ^t \tau^{\alpha-1}  || E_{\alpha, \alpha}(A,\tau)|| || \nu(t-\tau)|| d\tau =0$, and the claims follows.

\proofend

\begin{remark}

(i) When $\nu \equiv 0$ we recover \cite[Theorem 5]{C.Asym} and when $Q \equiv 0$ we recover the sufficient part of \cite[Theorem 13]{C.Perron}.

(ii) By the form of the proof, we cannot get an estimate of the speed or the order of the convergence. This drawback is shared in Cong et al. series of papers.

(iii) Note that the same arguments can be stated for complex valued solutions. In particular, $\mathcal{C}([0,\infty);\mathbb{C}^n)$ and its subset $\mathcal{C}_{\infty}$ are also Banach spaces.

\end{remark}

The following corollary generalizes  \cite[Theorem 6]{C.Asym} and  \cite[Theorem 8]{C.Asym} and shows, in its first claim for $T=0$, that hypothesis (\ref{Q}) is not empty.

\begin{corollary}
Under the same hypotheses of Theorem 2.1 and if there exists $T\geq 0$ such that $\sup_{t\geq T} ||Q(t)|| < \frac{1}{C(\alpha, A)}$, where $C(\alpha, A)  := \sup_{t\geq 0}  \int_0 ^t  \tau^{\alpha-1} ||  E_{\alpha, \alpha}(A,\tau) || d\tau < \infty$,  instead of condition (\ref{Q}), claims (i) and (ii) hold. In particular, if $\lim_{t \rightarrow \infty} ||Q(t)||=0$ claims (i) and (ii) hold.
\end{corollary}
\proof

By condition (\ref{QE}), we have $\sup_{t\geq 0}  || \int_0 ^t \tau^{\alpha-1}  E_{\alpha, \alpha}(A,\tau) d\tau || = C(\alpha, A) < \infty$ (\cite[Theorem 3]{C.Asym}). If $\sup_{t\geq 0} ||Q(t)|| < \frac{1}{C(\alpha, A)}$, we conclude that $\sup_{t\geq 0} || \int_0 ^t \tau^{\alpha-1}  E_{\alpha, \alpha}(A,\tau)  Q(t - \tau)d\tau || < 1$.

On the other hand, if  $\lim_{t \rightarrow \infty} ||Q(t)||=0$, then there exists $T$ such that $ \sup_{t\geq T} || \int_0 ^t \tau^{\alpha-1}  E_{\alpha, \alpha}(A,\tau)  [Q(t - \tau)]|| < 1$. Following the reasoning of the proof of Theorem 2.1, we have for $\xi \in \mathcal{C}_\infty$
\begin{equation*}
\begin{aligned}
\xi(t) = & E_{\alpha}(A,t) x(0) + \int_0 ^T (t-\tau)^{\alpha-1}  E_{\alpha, \alpha}(A, t-\tau)  Q(\tau) \xi(\tau) d\tau \\
      & + \int_T^t \tau^{\alpha-1}  E_{\alpha, \alpha}(A,\tau)  Q(t-\tau) \xi(t-\tau) d\tau  + \int_0 ^t \tau^{\alpha-1}  E_{\alpha, \alpha}(A,\tau)  \nu(t-\tau) d\tau.
\end{aligned}
\end{equation*}

The first integral of right-hand side converges to zero by applying \cite[Property 15]{FBL} since $||\int_0 ^T (t-\tau)^{\alpha-1}  E_{\alpha, \alpha}(A, t-\tau)  Q(\tau) \xi(\tau) d\tau || \leq C_E \int_0 ^T (t-\tau)^{\alpha-1}   ||Q(\tau) \xi(\tau) ||d\tau $. The rest follows from similar arguments of the proof of Theorem 2.1. Alternatively this can be proved with similar arguments of \cite[Theorem 8]{C.Asym}, where it is required to define a norm $||\cdot||_\beta$ which is equivalent to $||\cdot||_\infty$, because the additional terms associated to $\nu$ are canceled in subtracting the equations.

\proofend

\begin{remark}
Note, for control applications, that condition (\ref{Q}) or $\sup_{t} ||Q(t)|| < \frac{1}{C(\alpha, A)}$ depends on $A$ (more specifically on the magnitude of its eigenvalues) and $\alpha$. By choosing $A, \alpha$ we expand the range of matrices $Q$ assuring convergence (at least in principle, since the result is just a sufficient statement).
\end{remark}

\subsection{Comparison property}

The following result generalizes \cite[Theorem 5]{C.Asym} in the sense that matrix $A$ can now be a time varying function and the condition on $A(\cdot)$ is given as an inequality rather than an equality of type $A + Q(t) $. It can be seen as a comparison result generalizing that of \cite[Theorem 4]{FS}.

\begin{theorem}
Consider the following Caputo system
\begin{equation}
D^{\alpha} x(t) = A(t)x(t) + \nu(t)
\end{equation}
with $\alpha \in (0,1)$, $x \in \mathbb{R}^n$ and $\nu : [0, \infty) \rightarrow \mathbb{R}^n$  a bounded continuous function. Let $x_0 \in\mathbb{R}^n $ be the arbitrary initial condition. Suppose that $A: [0, \infty) \rightarrow \mathbb{R}^{n \times n}$ is a differentiable matrix function  satisfying $A(t) \leq -\epsilon I + Q(t) $ where $I$ is the identity matrix.

(i) If $\sup_{t} ||Q(t)|| < \mu < \epsilon$, then the solution $x$ is a bounded continuous function.

(ii) If $Q: [0, \infty) \rightarrow \mathbb{R}^{n \times n}$ is a bounded continuous matrix function such that $q = \sup_{t\geq 0} | \int_0 ^t \tau^{\alpha-1}  E_{\alpha, \alpha}(-\epsilon t^{\alpha})  [\lambda_M(t - \tau)]| d\tau < 1$, where $\lambda_M(t)$ is the largest eigenvalue of $Q(t)$ and $\nu \equiv 0$, then $\lim_{t \rightarrow \infty} x(t) = 0$.

\end{theorem}

\proof

(i)  We have
\begin{equation*}
\begin{aligned}
\> \>  1/2 D^{\alpha} [x^Tx]  \> \> \> \> \> \> \> \>  \leq {} & x^T D^{\alpha} x  =  x^TAx + x^T \nu  \\
   & \leq -\epsilon x^Tx + x^TQx + x^T \nu \leq -\epsilon x^Tx + \mu x^Tx + x^T \nu
\end{aligned}
\end{equation*}
where the first inequality is due to \cite[Lemma 1]{cuadr, cuadrt} which can be applied since $\nu$ and $A$ are differentiable functions, whereby $x$ is a differentiable function \cite[Property 12]{FS}.

Given that $Q$ and $\nu$ are bounded functions, the set $\Omega = \{ x \in \mathbb{R}^n |  \epsilon x^Tx - \mu x^Tx - x^T \nu \leq 0 \}$ is bounded (otherwise the positive term $(\epsilon- \mu) x^Tx$ would dominate leading to a contradiction). Whereby $\bar{\Omega}$ is compact and therefore $x$ is bounded by \cite[Theorem 8]{FS}.

(ii) We have
\begin{equation*}
1/2 D^{\alpha} [x^Tx] \leq x^T D^{\alpha} x  \leq  x^TAx \leq x^T(-\epsilon I + Q)x \leq -\epsilon x^Tx + \lambda_M(t) x^Tx
\end{equation*}
where the third and fourth inequalities are due to the hypothesis on $A(t)$. By using Theorem 2.1 on the equation $1/2 D^{\alpha} y = - \epsilon y + \lambda_M(t) y$ we conclude that $\lim_{t \rightarrow \infty} y(t) = 0$. By comparison principle \cite[Lemma 1 ]{FS}, we conclude that  $\lim_{t \rightarrow \infty} x(t) = 0$.

\proofend

\begin{remark}
If $||Q||$ converges to zero, then $\lambda_M$ also converges to zero and applying Corollary 2.1, we have $\lim_{t \rightarrow \infty} x = 0$. Similarly if $\sup_{t} ||Q(t)|| < \frac{1}{C(\alpha, A)}$  for an induced matrix norm (for an induced matrix norm, the bound of the matrix norm bounds the largest eigenvalue).

\end{remark}

\subsection{Extensions for $\alpha \geq 1$ and other types  of derivatives}
\textbf{    }

\textbf{    }

The starting point of Cong et al. is the analytic-implicit solution of  equation (\ref{Linear1}). For $ 1 < \alpha \leq 2$, it takes the form \cite[Theorem 5.15]{Kilbas}
\begin{equation}\label{alpha.gret}
x(t)= E_{\alpha} (t^{\alpha} A) x_0  + t E_{\alpha, 2} (t^{\alpha} A) \dot{x}_0 +  \int_0 ^t \tau^{\alpha-1}  E_{\alpha, \alpha}(A,\tau)  Q(t-\tau) x(t-\tau) d\tau
\end{equation}
where $x_0, \dot{x}_0 \in \mathbb{R}^n$ are initial conditions associated to the function $x$ and its derivative, respectively.

This can be directly proved by applying Laplace Transform. Moreover, it can be generalized for any order of derivation by the same procedure. However, for $\alpha \geq 2$, there is no condition for a stable constant matrix, namely $\{ \lambda \in \mathbb{C} \setminus \{0\}: |\arg (\lambda)| > \frac{\alpha \pi}{2} \} = \phi $ for $\alpha \geq 2$.

In addition, one can verify uniqueness of the solution (see for instance  \cite[Theorem 6.5]{Diethelm}, \cite[Theorem 3.25]{Kilbas} and the global version in \cite{Sin}) provided that $Q$ is a bounded function. Now we can formulate the following extension to \cite[Theorem 5]{C.Asym}, which is a generalization of \cite[Theorem 4]{FS}.

\begin{theorem}
Consider the following Caputo system
\begin{equation} \label{gen}
D^{\alpha} x(t) = (A + Q(t))x(t)  + \nu(t)
\end{equation}
with $0 < \alpha < 2$, $x \in \mathbb{R}^n$ and  $\nu : [0, \infty) \rightarrow \mathbb{R}^n$ a bounded continuous function ($ ||\nu||_\infty < \nu_0$). Let $x_0, \dot{x}_0 \in\mathbb{R}^n $ be  arbitrary initial conditions. Suppose that $Q : [0, \infty) \rightarrow \mathbb{R}^{n \times n}$ is a continuous matrix function such that $q = \sup_{t\geq 0} || \int_0 ^t \tau^{\alpha-1}  E_{\alpha, \alpha}(A,\tau)  [Q(t - \tau)]|| < 1$. Assume that $A \in \mathbb{R}^{n \times n}$ is a constant matrix such that $\sigma(A) \subset \Lambda^s_\alpha = \{ \lambda \in \mathbb{C} \setminus \{0\}: |\arg (\lambda)| > \frac{\alpha \pi}{2} \}$. Then,

(i) $x$ is bounded continuous function.

(ii) If in addition $ \lim_{t \rightarrow \infty} || \int_0 ^t \tau^{\alpha-1}  E_{\alpha, \alpha}(A,\tau)  \nu (t-\tau) d\tau || = 0$ then $\lim_{t \rightarrow \infty} ||x(t)||=0$. In particular, if $\nu$ converges to zero, then $\lim_{t \rightarrow \infty} ||x(t)||=0$.

\end{theorem}

\proof
(i) By applying Laplace Transform we conclude that the solution must hold
\begin{equation*}
x(t)= E_{\alpha} (t^{\alpha} A) x_0  + t E_{\alpha, 2} (t^{\alpha} A) \dot{x}_0 +  \int_0 ^t \tau^{\alpha-1}  E_{\alpha, \alpha}(A,\tau) [ Q(t-\tau) x(t-\tau) + \nu(t-\tau)] d\tau.
\end{equation*}

We define the Lyapunov-Perron operator by
\begin{equation*}
\begin{aligned}
\mathcal{T}_{x_0, \dot x_0} (\xi) (t) := & E_{\alpha} (t^{\alpha} A) x_0  + t E_{\alpha, 2} (t^{\alpha} A) \dot{x}_0  \\ & + \int_0 ^t \tau^{\alpha-1}  E_{\alpha, \alpha}(A,\tau) [ Q(t-\tau) \xi(t-\tau) + \nu(t-\tau)] d\tau.
\end{aligned}
\end{equation*}

Choose $\epsilon > 0$  arbitrary. By taking $\xi, \tilde{\xi} \in B_ \infty (0, \epsilon) = \{ \xi \in \mathcal{C}([0, \infty); \mathbb{R}^n) : ||\xi||_{\infty} \leq \epsilon \} $, we have
\begin{equation*}
||\mathcal{T}_{x_0, \dot x_0} (\xi) - \mathcal{T}_{x_0, \dot x_0} (\tilde{\xi}) ||_\infty \leq q ||\xi - \tilde{\xi}||_\infty.
\end{equation*}

Since $\sigma(A) \subset \Lambda^s_\alpha$, the initial conditions terms are bounded and decaying to zero. Since $\nu$ is bounded and $||\int_0 ^t \tau^{\alpha-1} E_{\alpha, \alpha}(A,\tau) d\tau ||< \infty$ by using   \cite[Theorem 3]{C.Asym} since $\sigma(A) \subset \Lambda^s_\alpha$ , we have $||\int_0 ^t \tau^{\alpha-1} E_{\alpha, \alpha}(A,\tau) \nu(t-\tau) d\tau)||< \infty$.  Hence,
\begin{equation*}
||\mathcal{T}_{x_0, \dot x_0} (\xi)||_\infty  < \infty.
\end{equation*}

Since $q < 1$, $\mathcal{T}_{x_0, \dot x_0} $ restricted to $\mathcal{C}_{\infty} =  \{ \xi \in \mathcal{C}([0, \infty); \mathbb{R}^n) : ||\xi||_{\infty} < \infty \} $ is a contraction auto map in a Banach space, whereby it has  a unique fixed point. Since the unique solution to  $D^{\alpha} x = (A + Q(t))x$ can be written as $x(t)= E_{\alpha} (t^{\alpha} A) x_0 + t E_{\alpha, 2} (t^{\alpha} A) \dot{x}_0 + \int_0 ^t \tau^{\alpha-1}  E_{\alpha, \alpha}(A,\tau)  [Q(t-\tau) \xi(t-\tau) + \nu(t-\tau) ] d\tau$ we conclude that $x=\xi$, whereby $x$ is a bounded continuous function and $\limsup_{t \rightarrow \infty} ||x(t)||< \infty$.

(ii) Define $\epsilon := \frac{ 1- q}{2q +1} \limsup_{t \rightarrow \infty} ||x(t)||$. Let $T$ be such that $\sup_{t \geq T} |x(t)| \leq \limsup_{t \rightarrow \infty} ||x(t)|| + \epsilon $, then
\begin{equation*}
\limsup_{t \rightarrow \infty} ||x(t)|| \leq (\limsup_{t \rightarrow \infty} ||x(t)|| + \epsilon) \sup_{t\geq 0} ||\int_T ^t \tau^{\alpha-1}  E_{\alpha, \alpha}(A,\tau)  Q(t-\tau) d\tau||
\end{equation*}
where we use that the condition $\sigma(A) \subset \Lambda^s_\alpha$ implies $\lim_{t \rightarrow \infty}  E_{\alpha} (t^{\alpha} A) x_0  = \lim_{t \rightarrow \infty}  t E_{\alpha, 2} (t^{\alpha} A) \dot{x}_0 = 0$ (see e.g. \cite[Lemma 3(i)]{Tuan}) and the hypothesis $\lim_{t \rightarrow \infty} ||\int_T ^t \tau^{\alpha-1}  E_{\alpha, \alpha}(A,\tau)  \nu(t-\tau) d\tau||=0$.

Therefore,
\begin{equation*}
\limsup_{t \rightarrow \infty} ||x(t)|| \leq  \limsup_{t \rightarrow \infty} ||x(t)|| \frac {2q + q^2}{2q +1} < \limsup_{t \rightarrow \infty} ||x(t)||
\end{equation*}
which is a contradiction if $\limsup_{t \rightarrow \infty} ||x(t)|| \neq 0$. Hence $\limsup_{t \rightarrow \infty} ||x(t)||=0$, and thus $\lim_{t \rightarrow \infty} ||x(t)|| = 0$.

The case when $\nu$ converges to zero is similar to the proof of Theorem 2.1.(ii)

\proofend




More generally, one can extend Theorem 2.3 to other type of fractional derivatives provided that, by applying Laplace transform, one have  that (i) $\mathcal{L}[D^{\alpha} f](s)= s^{\alpha} f(s) - F(0,s)$ for any continuous function $f$ where $F(0,s)$ is an initial condition term depending on what type of fractional derivative is used. Whereby the implicit solution to (\ref{gen}) can be written in Laplace Domain as $x(s) = (s^{\alpha}- I)^{-1}F(0,s) + (s^{\alpha}- I)^{-1} [Q x](s) $ (ii) the term $\mathcal{L}^{-1} [F(0,s)](t)$ converges to zero provided that $\sigma(A) \subset \Lambda^s_\alpha = \{ \lambda \in \mathbb{C} \setminus \{0\}: |\arg (\lambda)| > \frac{\alpha \pi}{2} \}$ and (iii) uniqueness of the solution holds for system (\ref{gen}). For instance, Riemann-Liouville derivative satisfies the above three conditions under mild assumptions (see \cite[\S 2.8]{Podlubny} for (i), \cite[\S 3]{Kilbas} for (iii), meanwhile for (ii) one always can intersect the stable set of matrices with matrices $A$ such that $\sigma(A) \in \Lambda^s_\alpha $).

We now extend the first Lyapunov Theorem in its stability-instability part for $\alpha > 1$. It will appear as a direct consequence of a more general result. When $\alpha<1$, the stability was proved in \cite[Theorem 5]{C.Lin} and the instability in \cite[Theorem 8]{C.Instability}. This theorem also is more general that the result of  \cite[Theorem 4]{Tuan}, since it considers the unstable case.

Though the function $h(x)=Ax + f(x)$ (where $h(0)=0$) does not depend on $\dot x$, we remark that for $\alpha > 1$ the equilibrium point must consider this variable; it cannot be calculated by  just asking for those $x$ such that $h(x)=0$. By uniqueness and by a closed inspection of equation (\ref{alpha.gret}) with $Qx$ replaced by $f(x)$, we conclude that if $(x_0, \dot{x}_0)=(0,0)$ then the consistent solution is $x \equiv 0$, thus $(x^T, \dot{x}^T)^T=(0,0)^T$ is equilibrium point.

\begin{theorem}
Consider the following Caputo system with  $0<\alpha<2$
\begin{equation}\label{Syst.aut}
D^{\alpha} x = Ax + f(x)
\end{equation}
with $x: [0, \infty) \rightarrow \mathbb{R}^n $ and $f: [0, \infty) \rightarrow \mathbb{R}^n$ a locally Lipschitz function with Lipschitz  constant such that $\lim_{r \rightarrow 0^+} L(r) =0$

(i) If $\sigma(A) \subset \Lambda^s_\alpha =\{ \lambda \in \mathbb{C} \setminus \{0\}: |\arg (\lambda)| > \frac{\alpha \pi}{2} \}$ then $(x, \dot{x})^T=(0,0)^T$ is a locally asymptotically  stable equilibrium point.

(ii) If $A$ has at least one eigenvalue in $\Lambda^u_\alpha  = \{ \lambda \in \mathbb{C} \setminus \{0\}: |\arg (\lambda)| < \frac{\alpha \pi}{2} \}$ then $(x^T, \dot{x}^T)^T=(0,0)^T$ is an unstable equilibrium point.
\end{theorem}

\proof
(i) To carry out the arguments of Cong et al. for $\alpha>1$ one must prove first that $ \sup_{t\geq 0} || \int_0 ^t \tau^{\alpha-1}  E_{\alpha, \alpha}(A,\tau)|| d\tau $ is bounded for those $\alpha$ provided that $\sigma(A) \subset \Lambda^s_\alpha $. We state the following reasoning.

In the scalar case $\int |\tau^{\alpha-1} E_{\alpha,\alpha}(A,\tau)| d\tau < \infty$ is equivalent to determine if $|\tau^{\alpha-1} E_{\alpha,\alpha}(A,\tau)| \in \mathcal{L}^1$. Which is related to determine if the system is BIBO stable. \cite[Theorem 2]{Malti} states BIBO stability for $\alpha > 1$  by proving conditions to obtain impulse response $h$ in $\mathcal{L}^1$ (Note that $h \in  \mathcal{L}^1$ implies BIBO stability but the converse is not true).

We have $\tau^{\alpha-1} E_{\alpha,\alpha}(A,\tau) = \mathcal{L}^{-1} [(s^{\alpha}I-A)^{-1}]$. Since $\sigma(A) \subset \Lambda^s_\alpha$, the following transfer function is BIBO stable for any $c, b \in \mathbb{R}^n$ by \cite[Theorem 1]{Malti}
\begin{equation*}
h(t)= c^T \mathcal{L}^{-1} [(s^{\alpha}I-A)^{-1}](t) b.
\end{equation*}

Moreover by \cite[Theorem 2]{Malti} $h(t) \in \mathcal{L}^1$ since $\hat{h}= \mathcal{L}[h](s)$ has positive relative degree. In particular, choosing $c, b$ as $e_i, e_j$ elements of the canonical base of $\mathbb{R}^n$, it follows that every element of matrix   $\mathcal{L}^{-1} [(s^{\alpha}I-A)^{-1}]$ belongs to $\mathcal{L}^1$.

Since $||A||_1= \sup_{i,j} |a_{ij} | < \sum_{i,j =1}^n |a_{ij} | $, we have $||\tau^{\alpha-1} E_{\alpha,\alpha}(A,\tau)||_1  \in\mathcal{L}^1$ . Note that the chosen norm is equivalent to others norms such as the Frobenius norm and $||\cdot||_2$, whereby it is compatible with the standard norm on  $\mathbb{R}^n$.

Since $ \sup_{t\geq 0}  \int_0 ^t \tau^{\alpha-1} || E_{\alpha, \alpha}(A,\tau)|| d\tau < \infty$ and $||f(x)|| \leq L(r) ||x||$, we can choose a sufficiently small $r$ so that  $ \sup_{t\geq 0}  \int_0 ^t \tau^{\alpha-1} || E_{\alpha, \alpha}(A,\tau) L(r) \xi(t-\tau)|| d\tau < 1$ for $\xi \in B_ {\infty} (0, r)$, so the convergence proof follows on the same way as in Theorem 2.3.

(ii) Suppose that $x=0, \dot{x}=0$ is an stable point. Then for all $ \epsilon>0$ there exists $ \delta$ such that if $(x_0^T, \dot{x}_0^T)^T \in B_{\mathbb{R}^n}(0, \delta)$ then $(x^T, \dot{x}^T)^T \in B_{\mathbb{R}^n}(0,  \epsilon)$. In particular, if $(x_0^T, 0^T)^T \in B_{\mathbb{R}^n}(0, \delta)$ then $x \in B_{\mathbb{R}^n}(0, \epsilon)$. We prove that this last case cannot occur, whereby $x=0, \dot{x}=0$ is an unstable point.

By a Jordan canonical transformation and without loss of generality, we can suppose $A$ a diagonal matrix of eigenvalues with its first simple eigenvalue $\lambda_1$ belonging to $\Lambda^u_\alpha $. Therefore, the solution to (\ref{Syst.aut}) with initial condition $(x(0)^T, \dot{x}(0)^T)^T=(x^1_0,0, \ldots, 0)^T$ can be expressed for the first component as
\begin{equation*}
x^1(t)= E_{\alpha}(\lambda_1 t^{\alpha})x^1_0 +  \int_0 ^t \tau^{\alpha-1}  E_{\alpha, \alpha}(A,\tau) f^1(x(t-\tau)) d\tau.
\end{equation*}
(if all the unstable eigenvalue are repeated, we use the not null initial condition in the equation corresponding to the last equation in (\ref{recursive}) obtained also by a canonical Jordan transformation, with $\tilde{\nu} \equiv f^1$).

By \cite[Lemma 4]{C.Instability} -- a result directly deduced from \cite[Theorems 1.3, 1.4]{Podlubny} which has validity for $0<\alpha<2$ -- we have
\begin{equation*}
\begin{aligned}
\lim_{t\rightarrow \infty}  & \int_0 ^t \frac{\tau^{\alpha-1} E_{\alpha, \alpha}(\lambda \tau^\alpha) f^1(x(t-\tau))}{E_{\alpha}(\lambda_1 t^{\alpha})}  d\tau = \lambda_1 ^{\frac{1-\alpha}{\alpha}}  \int_0 ^t \exp(-\lambda_1 ^{\frac{1-\alpha}{\alpha}} \tau) f^1(x(t-\tau))d\tau.
\end{aligned}
\end{equation*}

Since $E_{\alpha}(\lambda_1 t^{\alpha})$ diverges and $||x^1||_{\infty} < \epsilon$ by our assumption of stability, necessarily
\begin{equation}\label{init.cond}
x^1_0 = - \lambda_1 ^{\frac{1-\alpha}{\alpha}}  \int_0 ^t \exp(-\lambda_1 ^{\frac{1-\alpha}{\alpha}} \tau) f^1(x(t-\tau))d\tau.
\end{equation}

Defining,
\begin{equation*}
\begin{aligned}
\mathcal{T \xi}^1 :=  & \int_0 ^t \tau^{\alpha-1} E_{\alpha, \alpha}(\lambda \tau^\alpha) h^1(\xi(t-\tau)) d\tau \\ & - \lambda_1 ^{\frac{1-\alpha}{\alpha}}  E_{\alpha}(\lambda_i t^{\alpha})\int_0 ^t \exp(-\lambda_1 ^{\frac{1-\alpha}{\alpha}} \tau) h^1(\xi(t-\tau))d\tau
\end{aligned}
\end{equation*}
and $i=2, \ldots, n$
\begin{equation*}
\mathcal{T \xi}^i :=  \int_0 ^t \tau^{\alpha-1} E_{\alpha, \alpha}(\lambda \tau^\alpha) f^i(\xi(t-\tau)) d\tau
\end{equation*}
we observe that the solution to (\ref{Syst.aut}) with initial condition $(x(0)^T, \dot{x}(0)^T)=(x^1_0,0, \ldots, 0)^T$ must be a fixed point of $\mathcal{T}$.

Using, \cite[Lemma 3]{C.Instability} -- a result directly deduced from \cite[Theorems 1.3, 1.4]{Podlubny} which has validity for $0<\alpha<2$ -- ,  we have
\begin{equation*}
||\mathcal{T \xi}^1 || \leq K(\alpha, \lambda_1) ||f||_\infty
\end{equation*}
and using part (i), we have for $i=2, \ldots, n$
\begin{equation*}
||\mathcal{T \xi}^i || \leq C(\alpha, A) ||f||_\infty.
\end{equation*}

Choosing $\epsilon$ such that $(K(\alpha, \lambda_1) + C(\alpha, A)) l(\epsilon) <1$, we have
\begin{equation*}
||x||_\infty = ||\mathcal{T}x||_\infty  \leq C(\alpha, A) l(\epsilon) ||x||_\infty
\end{equation*}
which is a contradiction for $x^1_0 \neq 0$. Then $(x^T, \dot{x}^T)^T=(0,0)^T$ is an unstable point.

\proofend

\begin{remark}
(i) First Lyapunov Theorem appears by Taylor expansion around the origin, namely $f(x) = J_f(0) \> x + p(x(t))$ and noting that $p$ is locally Lipschitz around the origin with Lipschitz constant such that $\lim_{r \rightarrow 0^+} L(r) =0$. If the Jacobian of $f$ at zero $J_f(x=0)=J_f(0)$ holds that $\sigma(J_f(0)) \subset \Lambda^s_\alpha $ then local asymptotic stability of the origin $(x, \dot{x})=(0,0)$ follows and instability if at least one eigenvalue belongs to $\Lambda^u_\alpha $. We recall that even for integer order systems, the stability cannot be asserted from this local analysis if  there exist eigenvalues such that $\{ \lambda \in \mathbb{C} \setminus \{0\}: |\arg (\lambda)| = \frac{\alpha \pi}{2} \}$. Note also that instability does not necessarily imply that the solution diverges as is shown in \cite[Proposition 6]{C.Perron}, where  is proved that there exists initial conditions (in our case, holding equation (\ref{init.cond})) such that the solution remains bounded for unstable systems.

(ii) From the proof, we conclude that if $\sigma(A) \subset \Lambda^s_\alpha $, then  $\tau^{\alpha-1} E_{\alpha,\alpha}(A,\tau)$ is $\mathcal{L}^1$. Thus, the same  statement of Corollary 2.1 can be extended to $\alpha \geq 1$ because the proof is similar.

\end{remark}

\subsection{Extensions for piece-wise continuous functions}
\textbf{    }

\textbf{    }

$Q(t)$ in Cong et al. argumentation is always required to be a continuous function. We extend the results for the case when $Q(t)$ is only piece-wise continuous, namely one can define intervals $[T_i, T_{i+1})$ for $i \in \mathbb{N}$ such that $Q(t)$ is continuous in $[T_i, T_{i+1})$. This problem is relevant in switching systems, for instance $Q(t)$ could take  stable constant matrix values in $[T_i, T_{i+1})$.

\begin{theorem}
Consider the following Caputo system
\begin{equation} \label{pwc}
D^{\alpha} x(t) = (A + Q(t))x(t) + \nu(t)
\end{equation}
where $0<\alpha < 2$. Let $x_0, \dot{x}_0 \in\mathbb{R}^n $ be  arbitrary initial conditions and $\nu : [0, \infty) \rightarrow \mathbb{R}^n$ a bounded continuous function ($ ||\nu||_\infty < \nu_0$). Suppose that $Q : [0, \infty) \rightarrow \mathbb{R}^{n \times n}$ is a piece-wise bounded continuous matrix function such that $q = \sup_{t\geq 0} || \int_0 ^t \tau^{\alpha-1}  E_{\alpha, \alpha}(A,\tau)  Q(t - \tau)|| d\tau< 1$. Assume that $A \in \mathbb{R}^{n \times n}$ is such that $\sigma(A) \subset \Lambda^s_\alpha = \{ \lambda \in \mathbb{C} \setminus \{0\}: |\arg (\lambda)| > \frac{\alpha \pi}{2} \}$. Then, $x$ is a bounded continuous function. If in addition $ \lim_{t \rightarrow \infty} || \int_0 ^t \tau^{\alpha-1}  E_{\alpha, \alpha}(A,\tau)  \nu (t-\tau) d\tau || = 0$, then $\lim_{t \rightarrow \infty} ||x(t)||=0$. In particular, if $\nu$ converges to zero, then $\lim_{t \rightarrow \infty} ||x(t)||=0$.
\end{theorem}

\proof
To apply Cong et al's methodology, one must prove when $Q(t)$ is piece-wise continuous that the solution to (\ref{pwc}) (i) has an analytic-implicit solution in terms of $E_\alpha$ functions to form the Lyapunov Perron-Operator and that (ii) this solution is the unique continuous function satisfying the given initial condition.

Part (i) follows by applying Laplace transform, since Laplace transform is still valid for piece continuous functions. The resultant expression is the same as before i.e.
\begin{equation}
x(t)= E_{\alpha} (t^{\alpha} A) x_0 + t E_{\alpha, 2} (t^{\alpha} A) \dot{x}_0 + \int_0 ^t (t-\tau)^{\alpha-1}  E_{\alpha, \alpha}(A,t-\tau) [ Q(\tau) x(\tau) + \nu(\tau)] d\tau.
\end{equation}

For (ii), first we note that equation (\ref{pwc}) with $x_0=0, \dot{x}_0=0$ has a unique solution $x \equiv 0$ in the first continuity period $[T_0, T_{1})$ since it is indistinguishable of a linear system for which it is known uniqueness (\cite[Theorem 3.25]{Kilbas}) and  $x \equiv 0$ satisfies both, the initial condition and the equation for Caputo derivative. Then one has $x(T_1^-)=0, \dot{x}_0(T_1^-)=0$ where $T_1$ is the end of the first continuity period, which gives a solution $x \equiv 0$ in the second continuity period. Recursively one has a solution $x \equiv 0$ to the equation (\ref{pwc}).

A contraction map argument on the operator
\begin{equation*}
\mathcal{T}_{x_0, \dot x_0} (\xi) (t) := \int_{T_i} ^{T_{i+1}} \tau^{\alpha-1}  E_{\alpha, \alpha}(A,\tau)  Q(t-\tau) \xi(t-\tau) d\tau
\end{equation*}
proves that the solution $x \equiv 0$ is unique in $\mathcal{C} [T_i,s)$ for all $i \in \mathbb{N}$ where $s$ is such that $||\int_{T_i} ^{s} \tau^{\alpha-1}  E_{\alpha, \alpha}(A,\tau)  Q(t-\tau) d\tau|| < 1$ and  we have dropped the  associated  term $\nu $ since it is canceled in a contraction map argument. Recursively, one can extend for the total interval $\mathcal{C} [T_i,T_{i+1})$.

Suppose now  that there exists $x \neq y$ in $\mathcal{C} [0,T)$
 with $x(0) = y(0)$ such that $D^\alpha x = (A + Q(t)) x$ and $D^\alpha y = (A + Q(t)) y$.

Therefore, $D^\alpha (x-y) = (A + Q(t)) (x-y)$ with $(x-y)(0)=0$, which gives as solution $x-y \equiv 0$. We conclude that the assumption $x-y \neq 0$ is a contradiction and thus, the solution is unique. The rest of the proof is identical to the proof of Theorem 2.3

\proofend

\subsection{Robustness for autonomous systems}
\textbf{    }

\textbf{    }

We extend \cite[Theorem 5]{C.Lin} by considering the autonomous systems $D^{\alpha} x =f(x)$ slightly perturbed around the equilibrium point $x=0$ in its Taylor expansion. We use the concepts of locally ultimately bound equilibrium point and locally attractive equilibrium point defined in \cite[\S 3.1]{FLT} meaning that the solutions starting near of the equilibrium point remain bounded or converge to the equilibrium point, respectively, as time goes to infinity.

\begin{theorem}
Consider the following Caputo system with $0 < \alpha < 2$
\begin{equation}
D^{\alpha} x(t) = (A + Q(t)) x(t) + f(x(t)) + \nu(t)
\end{equation}
with $x(t) \in \mathbb{R}^n$ for all $t \geq 0$, $\nu : [0, \infty) \rightarrow \mathbb{R}^n$ a bounded Lipschitz continuous function ($ ||\nu||_\infty < \nu_0$) and $f$ locally Lipschitz continuous around the origin with Lipschitz constant such that $\lim_{r \rightarrow 0^+} L(r) =0$ for the neighborhood around the origin  $B(0,r) \subset \mathbb{R}^n$ and such that $f(0)=0$. Suppose that $Q : [0, \infty) \rightarrow \mathbb{R}^{n \times n}$ is a bounded continuous matrix function  with
\begin{equation}\label{Q1}
q := \sup_{t\geq 0} || \int_0 ^t \tau^{\alpha-1}  E_{\alpha, \alpha}(A,\tau)  Q(t - \tau) d\tau || < 1
\end{equation}
and that $A \in \mathbb{R}^{n \times n}$ is a constant matrix such that
\begin{equation}
\sigma(A) \subset \Lambda^s_\alpha = \{ \lambda \in \mathbb{C} \setminus \{0\}: |\arg (\lambda)| > \frac{\alpha \pi}{2} \}
\end{equation}

Then,

(i) $x=0, \dot{x}=0$ is a locally ultimately bound equilibrium point.

(ii) If in addition $ \lim_{t \rightarrow \infty} || \int_0 ^t \tau^{\alpha-1}  E_{\alpha, \alpha}(A,\tau)  \nu (t-\tau) d\tau || = 0$ then $\lim_{t \rightarrow \infty} ||x(t)||=0$. In particular, if $\nu$ converges to zero, then $x=0,  \dot{x}=0$ is an attractive point equilibrium point.

\end{theorem}

\proof
We make the case $0<\alpha < 1$. As we have seen, the analysis of the  case $1 \leq \alpha < 2$ is similar. Define the Lyapunov-Perron operator as
\begin{equation*}
\begin{aligned}
\mathcal{T}_{x_0} (\xi) (t) := &  E_{\alpha} (t^{\alpha} A) x_0 + \int_0 ^t \tau^{\alpha-1}  E_{\alpha, \alpha}(A,\tau)  Q(t-\tau) \xi(t-\tau) d\tau  \\
      & + \int_0 ^t \tau^{\alpha-1}  E_{\alpha, \alpha}(A,\tau)  f(\xi(t-\tau)) d\tau  + \int_0 ^t \tau^{\alpha-1}  E_{\alpha, \alpha}(A,\tau)  \nu (t-\tau) d\tau.
\end{aligned}
\end{equation*}

Let $\epsilon > 0$ be a number to be defined later. By taking $\xi, \hat{\xi} (t) \in B_ {\infty} (0, \epsilon) = \{ \xi \in \mathcal{C}([0, \infty); \mathbb{R}^n) : ||\xi||_{\infty} \leq \epsilon \} $, and using that $\sigma(A) \subset \Lambda^s_{\alpha}$ we have $E_{\alpha} (t^{\alpha} A) x_0 $ is bounded and   $\int_0 ^t \tau^{\alpha-1} ||E_{\alpha, \alpha}(\tau ^{\alpha} A)|| d\tau = C(\alpha,A) < \infty$ (\cite[Theorem 3]{C.Asym}) and, by Lipschitz assumption, $||f(\xi) - f(\hat{\xi})||_\infty \leq L(\epsilon) ||\xi -\hat{\xi}||_\infty $, we have
\begin{equation*}
||\mathcal{T}_{x_0} \xi  - \mathcal{T}_{x_0} \hat{\xi} ||_\infty \leq q || \xi - \hat{\xi}||_\infty + C(\alpha,A) L(\epsilon) || \xi - \hat{\xi}||_\infty.
\end{equation*}

By taking $\hat{\xi} = 0$ and since $\sigma(A) \subset \Lambda^s_\alpha$, we get
\begin{equation*}
||\mathcal{T}_{x_0} (\xi) ||_\infty \leq C(\nu_0) + q \epsilon +  C(\alpha,A)L(\epsilon) \epsilon < \infty,
\end{equation*}
where $C(\nu_0)$ is a constant term depending on the bound $\nu_0$. Thus, $\mathcal{T}_{x_0} (B_ {\infty} (0, \epsilon)) \subset B_ {\infty} (0, C(\nu_0))$.

Since $\lim_{r \rightarrow 0^+} L(r) =0$, we can choose a sufficiently small $\epsilon < 1$ so that $\mu = q + C(\alpha,A)  L(\epsilon)  < 1$. Thus we conclude that
\begin{equation*}
||\mathcal{T}_{x_0} (\xi) - \mathcal{T}_{x_0} (\hat{\xi}) ||_\infty \leq \mu ||\xi - \hat{\xi}||_\infty.
\end{equation*}

Therefore, $\mathcal{T}_{x_0} $ restricted to the set $\mathcal{C}_{\infty} =  \{ \xi \in \mathcal{C}([0, \infty); \mathbb{R}^n) : ||\xi||_{\infty} < \infty \} $ is a contraction auto map in a Banach space,  whereby it has  a unique fixed point. Since the unique solution to  $D^{\alpha} x = F(t, x):= (A + Q(t))x + f(x(t)) + \nu(t)$ (the uniqueness follows from Lipschitz continuity of $F$ in the second variable) can be written as $x(t)= E_{\alpha} (t^{\alpha} A) x_0 + \int_0 ^t \tau^{\alpha-1}  E_{\alpha, \alpha}(A,\tau)  [Q(t-\tau) \xi(t-\tau) + f(x(t-\tau)) + \nu (t-\tau)] d\tau$, we conclude that $x=\xi$ when $x_0 \in B_{\mathbb{R}^n}(0,\epsilon)$, whereby $x$ is bounded continuous function since $\xi \in \mathcal{C}_{\infty}$.

(ii) Let  $x_0 \in B_{\mathbb{R}^n}(0,\epsilon)$. We have
\begin{equation*}
\begin{aligned}
x(t) = &   E_{\alpha}(A,t) x(0) + \int_0 ^t \tau^{\alpha-1}  E_{\alpha, \alpha}(A,\tau)  Q(t-\tau) \xi(t-\tau) d\tau  \\
      &  + \int_0 ^t \tau^{\alpha-1}  E_{\alpha, \alpha}(A,\tau)  f(\xi(t-\tau)) d\tau + \int_0 ^t \tau^{\alpha-1}  E_{\alpha, \alpha}(A,\tau)  \nu(t-\tau) d\tau
\end{aligned}
\end{equation*}
and taking norm we get
\begin{equation*}
\begin{aligned}
||x(t)|| \leq &   ||E_{\alpha}(A,t) x(0)|| + ||\int_0 ^t \tau^{\alpha-1}  E_{\alpha, \alpha}(A,\tau)  Q(t-\tau) \xi(t-\tau) d\tau||  \\
      &  + ||\int_0 ^t \tau^{\alpha-1}  E_{\alpha, \alpha}(A,\tau)  f(\xi(t-\tau)) d\tau||  + ||\int_0 ^t \tau^{\alpha-1}  E_{\alpha, \alpha}(A,\tau) \nu(t-\tau) d\tau||.
\end{aligned}
\end{equation*}

Define $\delta := \frac{ 1- \mu}{2\mu +1} \limsup_{t \rightarrow \infty} ||x(t)||$ (by part (i), we have that $\delta < \infty$) and $T$ such that $\sup_{t \geq T} ||x(t)|| \leq \limsup_{t \rightarrow \infty} ||x(t)|| + \delta $ (by part (i), such $T$ always exists since $x$ is bounded), then
\begin{equation*}
\begin{aligned}
\limsup_{t \rightarrow \infty} ||x(t)|| \leq & (\limsup_{t \rightarrow \infty} ||x(t)|| + \delta) (\sup_{t\geq 0} ||\int_T ^t \tau^{\alpha-1}  E_{\alpha, \alpha}(A,\tau)  [Q](t-\tau) d\tau|| \\ & + L(\epsilon) ||\int_0 ^t \tau^{\alpha-1}  E_{\alpha, \alpha}(A,\tau)  d\tau||)
\end{aligned}
\end{equation*}
where we used the hypothesis $ \lim_{t \rightarrow \infty} || \int_0 ^t \tau^{\alpha-1}  E_{\alpha, \alpha}(A,\tau)  \nu (t-\tau) d\tau || = 0$ and that since $\sigma(A) \subset \Lambda^s_\alpha$,  $ \lim_{t \rightarrow \infty} E_{\alpha} (t^{\alpha} A)=0$ (see e.g. \cite[Theorem 3]{C.Asym})).

Therefore,
\begin{equation*}
\limsup_{t \rightarrow \infty} ||x(t)|| \leq  \limsup_{t \rightarrow \infty} ||x(t)|| \frac {2\mu + \mu^2}{2\mu +1} < \limsup_{t \rightarrow \infty} ||x(t)||
\end{equation*}
which is a contradiction if $\limsup_{t \rightarrow \infty} ||x(t)|| \neq 0$. Hence, $\limsup_{t \rightarrow \infty} ||x(t)||=0$, whereby $\lim_{t \rightarrow \infty} ||x(t)|| = 0$.

Finally, note that
\begin{equation*}
0 \leq ||\int_0 ^t \tau^{\alpha-1}  E_{\alpha, \alpha}(A,\tau)  \nu(t-\tau) d\tau|| \leq \int_0 ^t \tau^{\alpha-1}  || E_{\alpha, \alpha}(A,\tau)|| || \nu(t-\tau)|| d\tau
\end{equation*}

Since $\sup_{t\geq 0}  \int_0 ^t  \tau^{\alpha-1} ||  E_{\alpha, \alpha}(A,\tau) || d\tau  = C(\alpha, A) < \infty$ (\cite[Theorem 3]{C.Asym}), we have $\tau^{\alpha-1} || E_{\alpha, \alpha}(A,\tau)|| \in \mathcal{L}^1$. Since $\nu$ is bounded and converges to zero, we can apply \cite[Lemma 2]{FS}, to conclude that $\int_0 ^t \tau^{\alpha-1}  || E_{\alpha, \alpha}(A,\tau)|| || \nu(t-\tau)|| d\tau =0$, and the claim follows.

\proofend

\begin{remark}
(i) By Corollary 2.1, instead of condition (\ref{Q1}) one can use $\sup_{t \geq T} ||Q(t)|| < \frac{1}{C(\alpha, A)}$ for some $T \geq 0$ or that  $\lim_{t \rightarrow \infty} ||Q(t)||=0$.

\end{remark}

From the above theorem, we state the following result

\begin{corollary}

Consider the following Caputo system
\begin{equation}
D^{\alpha} x = f(x,t)
\end{equation}
where $f(x,t)$ is a differentiable function with $f(0,\cdot)\equiv0$ and $\lim_{t \rightarrow \infty}  f(x,t) = g(x)$ with $g(0)=0$ and  the Jacobian of $g$ satisfying
\begin{equation*}
\sigma(J_g(0)) \subset \Lambda^s_\alpha = \{ \lambda \in \mathbb{C} \setminus \{0\}: |\arg (\lambda)| > \frac{\alpha \pi}{2} \}
\end{equation*}

Then $x=0$ is locally asymptotically stable.
\end{corollary}

\proof
Using that $\lim_{t \rightarrow \infty}  f(x,t) = g(x)$ and that $f(0,\cdot)\equiv0$, the local expansion around $x=0$ of $f$ must hold for any $t>0$ that $f(x, t)= (J_g(0) + Q(t)) x + q(x(t)) + \nu(x(t))$ where $Q, \nu$ vanishes asymptotically and $q$ is such that $g(x)=J_g(0) x + q(x)$ for a sufficiently small $x$. The claim follows from applying Theorem 2.6.(ii)
\proofend

\begin{remark}
(i) A trivial example of such $f$ is $f(x,t)= (A+Q(t))x$ with $A$ satisfying the assumption of Theorem 2.6 and $\lim_{ t \rightarrow \infty} Q(t)=0$. Another example is $f(x,t)= f(x) + u(t)g(x)$ when $u(t) \rightarrow 0$, and the Jacobian of $f$ holds condition (2.13). Similarly, if  $f(x,t)= f(x) + u(t)g(x)$ with $u(t)g(x) \rightarrow h(x)$, $f(0)=0$, $g(0)=0$ assuming the Jacobian of $f$ holds condition (2.13) and $h$ is locally Lipschitz around the origin with Lipschitz constant such that $\lim_{r \rightarrow 0^+} L(r) =0$ for the neighborhood around the origin  $B(0,r)$. In fact, since locally around $x=0$, $f(x)= A x + p(x)$, $u(t)g(x)=  u(t)Qx + u(t)p_g(x)$ we have locally around $x=0$ that $f(x,t)= (A + u(t)Q) + u(t)p_g(x) + p$ where $p$ is locally Lipschitz around the origin with Lipschitz constant such that $\lim_{r \rightarrow 0^+} L(r) =0$. Then $x=0$ is locally asymptotically stable point.

(ii) Actually, with the aid of the proof of Theorem 2.2, one just requires that $\lim_{t \rightarrow \infty}  f(x,t) = g(x,t)$ such that $g(x,t) \leq (A + Q(t))$ with $A$ satisfying condition (2.13) and $Q$  satisfying condition (2.12).
\end{remark}



\subsection{Mixed order systems}
\textbf{    }

\textbf{    }

Common to all developments of Cong et al. is that the same order of derivation is used for each component of vector function $x$ in its defining equation. This represents a serious modeling limitation. We present the generalization when different orders of derivation for each component of vector $x$ are allowed.

\begin{theorem}
Consider $\alpha_i > 0$ for $i = 1 , \ldots, n$. Define $D^{\alpha} x(t)$ a  vector of components $D^{\alpha_i} x_i(t)$ where $x(t) \in \mathbb{R}^n$ for all $t \geq 0$. Consider the following Caputo system,
\begin{equation}\label{S}
D^{\alpha} x(t) = A x(t)+ Q(t)x(t) + f(x(t)) + \nu(t)
\end{equation}
where $\nu : [0, \infty) \rightarrow \mathbb{R}^n$ is a bounded  continuous function ($ ||\nu||_\infty < \nu_0$) and $f$ is locally Lipschitz continuous around the origin with Lipschitz constant such that $\lim_{r \rightarrow 0^+} L(r) =0$ for the neighborhood around the origin  $B(0,r)$ and such that $f(0)=0$. Suppose $Q : [0, \infty) \rightarrow \mathbb{R}^{n \times n}$ is a continuous matrix . Let $x_0, \dot{x}_0 \in\mathbb{R}^n $ be arbitrary initial conditions at the initial time of the Caputo derivative. Suppose that $A \in \mathbb{R}^{n \times n}$ is a constant matrix such that
\begin{equation}\label{ww}
\sigma(A) \subset \{ \lambda \in \mathbb{C} \setminus \{0\}: |\arg (\lambda)| > \frac{\alpha_M \pi}{2} \}
\end{equation}
where $\alpha_M= \max_i \{ \alpha_i \}$.

(i) Suppose $f(x) \equiv 0, Q \equiv 0$. Then, $x$ is a bounded continuous function and if $\nu$ converges to zero, $x$ converges to zero.

(ii) Suppose $\nu \equiv 0, Q \equiv 0$. Then, $x$ converges to zero.

(iii) Suppose $f(x) \equiv 0$ and $Q$ such that $q := \sup_{t\geq 0} || \int_0 ^t \tau^{\alpha-1}  E_{\alpha, \alpha}(\tau)  Q(t - \tau) d\tau || < 1$ where matrix $E_{\alpha, \alpha}(\tau)$ is a diagonal matrix with  $(E_{\alpha, \alpha}(t))_{ii}:=E_{\alpha_i, \alpha_i}(\lambda_i t^{\alpha_i}) $ for $\lambda_i \in \sigma(A)$.  If in addition $ \lim_{t \rightarrow \infty} || \int_0 ^t \tau^{\alpha-1}  E_{\alpha, \alpha}(A,\tau)  \nu (t-\tau) d\tau || = 0$ then $\lim_{t \rightarrow \infty} ||x(t)||=0$. In particular, if $\nu$ converges to zero, then $\lim_{t \rightarrow \infty} ||x(t)||=0$.



.
\end{theorem}

\begin{proof}
(i) Consider a non singular constant matrix $T$ transforming $A$ in its Jordan normal form, i.e.  $T^{-1} A T = diag (A_1, \ldots, A_r)$  where $A_i= \lambda_i I_{d_i \times d_i} + \eta_i N_{d_i \times d_i}$ for $i=1, \ldots, r$ with $I_{d_i \times d_i}$ is the identity matrix of dimension ${d_i \times d_i}$, $\lambda_i \in \sigma(A)$,  $\eta_i \in \{ 0, 1 \}$, $\eta_i=0$ if eigenvalue  $\lambda_i$ is simple and $1$ otherwise, $d_i$ is the multiplicity of $\lambda_i$, $\sum_i d_i=n$ and
\begin{equation*}
 N_{d_i \times d_i} := \begin{bmatrix}
0  & 1    & 0      & \ldots       & 0   \\
0  & 0 & 1 & \ldots &   0 \\
\vdots & \vdots & \ddots & \ddots & \vdots     \\
0  &  0      & \ldots & 0 & 1 \\
0 & 0      &  \ldots      & 0     & 0
\end{bmatrix}_{d_i \times d_i}
\end{equation*}
(See \cite[\S 3.2]{C.Perron}).

Since the transformation $T$ is non singular and linear, system (\ref{S}) is equivalent to the system $T^{-1} D^{\alpha}Tz =D^{\alpha}z=  diag (A_1, \ldots, A_d)z +   T^{-1}Q(t)T z + T^{-1}\nu(t)  $ if $z:=T^{-1}x$, meaning that boundedness and convergence to zero of variables $x, z$ are preserved under such linear non singular transformation. We define $\tilde{\nu}(t) :=T^{-1}\nu(t)$.

The system in $z$ is decoupled in $r$ independent subsystems, each of them  having the following recursive structure. For generic eigenvalue $\lambda$ with multiplicity $d_\lambda$, we have
\begin{equation} \label{recursive}
\begin{cases}
D^{\alpha_1}z_1(t)= \lambda z_1(t) + \eta z_2(t) +  \tilde{\nu}_1(t) \\
D^{\alpha_2}z_2(t)= \lambda z_2(t) + \eta z_3(t) +  \tilde{\nu}_2(t) \\
\ldots \\
D^{\alpha_{d_\lambda -1}}z_{d_\lambda -1}(t)= \lambda z_{d_\lambda -1}(t) + \eta z_{d_\lambda}(t) +  \tilde{\nu}_{d_\lambda -1}(t) \\
D^{\alpha_{d_\lambda}}z_{d_\lambda}(t)= \lambda z_{d_\lambda}(t) +   \tilde{\nu}_{d_\lambda}(t).
\end{cases}
\end{equation}

Since $|\arg (\lambda)| > \frac{\alpha_m \pi}{2} \geq \frac{\alpha_i \pi}{2} $ for any $i=1, \ldots, n$, we can apply \cite[Proposition 5]{C.Perron} (or Theorem 2.1 for the scalar case) to  the last equation of (\ref{recursive}) whereby $z_{d_\lambda}$ is bounded and if $\tilde{\nu}$ converges to zero $z_{d_\lambda}$ converges to zero. For the equation in $z_{d_\lambda -1}$, since $z_{d_\lambda}$ is bounded, the same arguments prove that $z_{d_\lambda-1}$ is bounded and if $\tilde{\nu}$ converges to zero then $z_{d_\lambda}$  converges to zero and therefore $z_{d_\lambda-1}$ converge to zero. Recursively, one get that $z$ and hence $x$, are bounded continuous function and if  $\tilde{\nu}$ converges to zero then $z$ and hence $x$ converge to zero.

(ii)
For simplicity, suppose $0< \alpha_i \leq 1$. Define $P_j:= diag (1, \gamma, \ldots, \gamma^{d_j-1})$ for $j = 1, \ldots, r$ with $\gamma>0$ a constant number to be precised. Hence, by part (i) $P_j^{-1}A_j P_j = \lambda_j I_{d_j \times d_j} + \gamma_j N_{d_j \times d_j}$, and therefore $(TP)^{-1}A TP=  diag (\lambda_1 id_{d_1 \times d_1}, \ldots , \lambda_r id_{d_r \times d_r}) + diag (\gamma_1 N_{d_1 \times d_1}, \ldots , \gamma_r N_{d_r \times d_r}) $ with $\gamma_j \in \{ 0, \gamma \}$. Hence, under the transformation $y:=(TP)^{-1}z$, system (\ref{S}) will be equivalent to $D^{\alpha_i}z_{i}(t)= \lambda_i z_{i}(t) + \gamma_i z_{i+1}+ \tilde{f}_i(z(t)) $ for $i=1, \ldots, n$ where $\tilde{f} (z) :=(TP)^{-1}f(TPy)$. Note that eigenvalues with multiplicity $d$ appear repeated $d$ in the list $\lambda_1, \ldots, \lambda_n$ . Therefore, we can express the system as  $z_i(t)= E_{\alpha_i}(\lambda_i t^{\alpha_i})z_i(0) + \int_0 ^t \tau^{\alpha-1}  E_{\alpha_i, \alpha_i}(\lambda_i\tau^{\alpha_i}) [ \gamma_i z_{i+1} +  \tilde{f}_{i}(t-\tau)]  d\tau$.

Defining the diagonal matrices, $(E_{\alpha}(t))_{ii}:=E_{\alpha_i}(\lambda_i t^{\alpha_i}) $, $(E_{\alpha, \alpha}(t))_{ii}:=E_{\alpha_i, \alpha_i}(\lambda_i t^{\alpha_i})$, $(\tau^{\alpha-1})_{ii}:=\tau^{\alpha_i-1}$ and building matrix $N$ from blocks $ \gamma_i N_{d_i \times d_i}$ in the diagonal ($N$ is a constant matrix), system (\ref{S}) can be written in vector form as
\begin{equation}
z(t)= E_{\alpha}(t)z(0) +  \int_0 ^t \tau^{\alpha-1}  E_{\alpha, \alpha}(\tau) [N z(t-\tau) + \tilde{f}(z(t-\tau))]  d\tau.
\end{equation}

By hypothesis $ E_{\alpha}(t)$ converges to zero and $||t^{\alpha-1}  E_{\alpha, \alpha}(t)||_1 \in \mathcal{L}^1$ since $|\arg (\lambda_i)| > \frac{\alpha_m \pi}{2} \geq \frac{\alpha_i \pi}{2} $,  $t^{\alpha_i-1}  (E_{\alpha, \alpha})_{ii}(t)\in \mathcal{L}^1$ for $i=1, \ldots, n$.  Define for vector $\alpha$, $C(\alpha, A) := \sup_{t \geq 0} ||\int_0^t \tau^{\alpha-1}  E_{\alpha, \alpha}(\tau) d\tau || < \infty$. Define Lyapunov-Perron operator as
\begin{equation*}
\mathcal{T}_{x_0} (\xi) (t) := E_{\alpha}(t)\xi(0) +  \int_0 ^t \tau^{\alpha-1}  E_{\alpha, \alpha}(\tau) [N \xi(t-\tau) + \tilde{f}( \xi(t-\tau)) ]  d\tau.
\end{equation*}

Let $\epsilon > 0$ a number to be defined later. Consider $\xi, \hat{\xi} (t) \in B_ {\infty} (0, \epsilon) = \{ \xi \in \mathcal{C}([0, \infty); \mathbb{C}^n) : ||\xi||_{\infty} \leq \epsilon \} $. Fix number $\gamma$  such that that $||N|| C(\alpha, A)=1/2$. Using that $\int_0 ^t \tau^{\alpha-1} ||E_{\alpha, \alpha}(\tau ^{\alpha} A)|| d\tau < \infty$ and, by Lipschitz assumption, $||f(\xi) - f(\hat{\xi})||_\infty \leq L(\epsilon) ||\xi -\hat{\xi}||_\infty $, we have
\begin{equation*}
||\mathcal{T}_{x_0} \xi  - \mathcal{T}_{x_0} \hat{\xi} ||_\infty \leq C(\alpha, A) ||N|| || \xi - \hat{\xi}||_\infty + C(\alpha, A) L(\epsilon) || \xi - \hat{\xi}||_\infty.
\end{equation*}

Hence, by taking $\hat{\xi} = 0$, we have
\begin{equation*}
||\mathcal{T}_{x_0} (\xi) ||_\infty \leq C(\alpha, A)||N|| \epsilon +  C(\alpha, A) L(\epsilon) \epsilon := C < \infty
\end{equation*}
and thus, $\mathcal{T}_{x_0} (B_ {{\infty}} (0, \epsilon)) \subset B_ {\infty} (0, C)$. Since $\lim_{r \rightarrow 0^+} L(r) =0$, we can choose a sufficiently small $\epsilon >0 $ so that $\mu =  C(\alpha, A) ||N|| +  C(\alpha, A) L(\epsilon)= 1/2  +  C(\alpha, A) L(\epsilon) < 1$. Thus we conclude that
\begin{equation*}
||\mathcal{T}_{x_0} (\xi) - \mathcal{T}_{x_0} (\hat{\xi}) ||_\infty \leq \mu ||\xi - \hat{\xi}||_\infty
\end{equation*}

Therefore, $\mathcal{T}_{x_0} $ restricted to the set $\mathcal{C}_{\infty} =  \{ \xi \in \mathcal{C}([0, \infty); \mathbb{C}^n) : ||\xi||_{\infty} < \infty \} $ is a contraction auto map in a Banach space,  whereby it has  a unique fixed point. Since the unique solution to  $D^{\alpha} z = F(t, z):= Az + f(z(t))$ (the uniqueness follows from Lipschitz continuity of $F$ in the second variable, along the same line of arguments of \cite[Theorem 6.5]{Diethelm} or alternatively, using directly  \cite[Theorem 6.5]{Diethelm} since under Jordan canonical form, the system is decomposed in recursive scalar equation) can be written as $z(t)= E_{\alpha} (t^{\alpha} A) z_0 + \int_0 ^t \tau^{\alpha-1}  E_{\alpha, \alpha}(\tau)  f(z(t-\tau)) d\tau$, we conclude that $x=\xi$ when $x_0 \in B_{\mathbb{C}^n}(0,\epsilon)$, whereby $z$ is bounded continuous function since $\xi \in \mathcal{C}_{\infty}$.

Define $\delta := \frac{ 1- \mu}{2\mu +1} \limsup_{t \rightarrow \infty} ||x(t)||$ (by part (i), we have that $\delta < \infty$) and $T$ such that $\sup_{t \geq T} ||x(t)|| \leq \limsup_{t \rightarrow \infty} ||x(t)|| + \delta $ (by part (i), such $T$ always exists since $x$ is bounded), then
\begin{equation*}
\begin{aligned}
\limsup_{t \rightarrow \infty} ||x(t)|| \leq & (\limsup_{t \rightarrow \infty} ||x(t)|| + \delta) (\sup_{t\geq 0} ||\int_T ^t \tau^{\alpha-1}  E_{\alpha, \alpha}(\tau)  N  d\tau|| \\ & + L(\epsilon) ||\int_0 ^t \tau^{\alpha-1}  E_{\alpha, \alpha}(A,\tau)  d\tau|| )
\end{aligned}
\end{equation*}
and thus,
\begin{equation*}
\limsup_{t \rightarrow \infty} ||x(t)|| \leq  \limsup_{t \rightarrow \infty} ||x(t)|| \frac {2\mu + \mu^2}{2\mu +1} < \limsup_{t \rightarrow \infty} ||x(t)||
\end{equation*}
which is a contradiction if $\limsup_{t \rightarrow \infty} ||x(t)|| \neq 0$. Hence, $\limsup_{t \rightarrow \infty} ||x(t)||=0$, whereby $\lim_{t \rightarrow \infty} ||x(t)|| = 0$.

The proof for $0< \alpha_i < 2$ is similar using the ideas of the proof of Theorem 2.4.

(iii) For simplicity, suppose $0< \alpha_i < 1$. By using the transformation of part (ii), system (\ref{S})  will be equivalent to $D^{\alpha_i}z_{i}(t)= \lambda_i z_{i}(t) + \gamma_i z_{i+1}(t) + \sum_{j} \tilde{Q}_{ij}(t) z_j +   \tilde{\nu}_{i}(t)$ for $i=1, \ldots, n$ where $\tilde{Q}(t) :=(TP)^{-1}Q(t)TP$. Or  $z_i(t)= E_{\alpha_i}(\lambda_i t^{\alpha_i})z_i(0) + \int_0 ^t \tau^{\alpha-1}  E_{\alpha_i, \alpha_i}(\lambda_i\tau^{\alpha_i}) [\gamma_i z_{i}(t) + \sum_{j} \tilde{Q}_{ij}(t-\tau) z_j (t-\tau) + \tilde{\nu}_{i}(t-\tau)]  d\tau$.

Defining the diagonal matrices, $(E_{\alpha}(t))_{ii}:=E_{\alpha_i}(\lambda_i t^{\alpha_i}) $, $(E_{\alpha, \alpha}(t))_{ii}:=E_{\alpha_i, \alpha_i}(\lambda_i t^{\alpha_i}) $, the solution can be written in vector form as
\begin{equation}
z(t)= E_{\alpha}(t)z(0) +  \int_0 ^t \tau^{\alpha-1}  E_{\alpha, \alpha}(\tau) [Nz(t-\tau) + \tilde{Q}(t-\tau) z(t-\tau) + \tilde{\nu}(t-\tau)]  d\tau.
\end{equation}

Define Lyapunov-Perron  operator as
\begin{equation*}
\mathcal{T}_{x_0} (\xi) (t) := E_{\alpha}(t)\xi(0) +  \int_0 ^t \tau^{\alpha-1}  E_{\alpha, \alpha}(\tau) [N \xi(t-\tau) + \tilde{Q}( \xi(t-\tau)) + \nu(t-\tau))]  d\tau.
\end{equation*}

Let $\epsilon > 0$ a number to be defined later. Consider $\xi, \hat{\xi} (t) \in B_ {\infty} (0, \epsilon) = \{ \xi \in \mathcal{C}([0, \infty); \mathbb{C}^n) : ||\xi||_{\infty} \leq \epsilon \} $. Fix number $\gamma$  such that that $||N|| C(\alpha, A) + q < 1 $. Using that $\int_0 ^t \tau^{\alpha-1} ||E_{\alpha, \alpha}(\tau ^{\alpha} A)|| d\tau = C(\alpha, A) < \infty$, we have
\begin{equation*}
||\mathcal{T}_{x_0} \xi  - \mathcal{T}_{x_0} \hat{\xi} ||_\infty \leq C(\alpha, A) ||N|| || \xi - \hat{\xi}||_\infty + q || \xi - \hat{\xi}||_\infty
\end{equation*}
where we use that $||\int_0 ^t \tau^{\alpha-1}  E_{\alpha, \alpha}(\tau) [(TP)^{-1}Q(t-\tau)TP( \xi(t-\tau)]  d\tau|| = ||\int_0 ^t \tau^{\alpha-1}  E_{\alpha, \alpha}(\tau) [Q(t-\tau)( \xi(t-\tau)]  d\tau||$.

Hence, by taking $\hat{\xi} = 0$, we have
\begin{equation*}
||\mathcal{T}_{x_0} (\xi) ||_\infty \leq C (\nu_0) + C(\alpha, A)||N|| \epsilon +  q \epsilon := C < \infty
\end{equation*}
where $C(\nu_0)$ is a constant term depending on the bound $\nu_0$ of function $\nu$. Thus $\mathcal{T}_{x_0} (B_ \infty (0, \epsilon)) \subset B_ \infty (0, C)$. Since $\mu :=  C(\alpha, A) ||N|| + q < 1$, we conclude that
\begin{equation*}
||\mathcal{T}_{x_0} (\xi) - \mathcal{T}_{x_0} (\hat{\xi}) ||_\infty \leq \mu ||\xi - \hat{\xi}||_\infty.
\end{equation*}

Therefore, $\mathcal{T}_{x_0} $ restricted to the set $\mathcal{C}_{\infty} =  \{ \xi \in \mathcal{C}([0, \infty); \mathbb{C}^n) : ||\xi||_{\infty} < \infty \} $ is a contraction auto map in a Banach space,  whereby it has  a unique fixed point. Since the unique solution to  $D^{\alpha} z = F(t, z):= Az + f(z(t)$ (the uniqueness follows from Lipschitz continuity of $F$ in the second variable) can be written as $z(t)= E_{\alpha} (t^{\alpha} A) z_0 + \int_0 ^t \tau^{\alpha-1}  E_{\alpha, \alpha}(\tau)  f(z(t-\tau)) d\tau$, we conclude that $x=\xi$ when $x_0 \in B_{\mathbb{C}^n}(0,\epsilon)$, whereby $z$ is bounded continuous function since $\xi \in \mathcal{C}_{\infty}$.

Define $\delta := \frac{ 1- \mu}{2\mu +1} \limsup_{t \rightarrow \infty} ||x(t)||$ (by part (i), we have that $\delta < \infty$) and $T$ such that $\sup_{t \geq T} ||x(t)|| \leq \limsup_{t \rightarrow \infty} ||x(t)|| + \delta $ (by part (i), such $T$ always exists since $x$ is bounded), then
\begin{equation*}
\begin{aligned}
\limsup_{t \rightarrow \infty} ||x(t)|| \leq & (\limsup_{t \rightarrow \infty} ||x(t)|| + \delta) (\sup_{t\geq 0} ||\int_T ^t \tau^{\alpha-1}  E_{\alpha, \alpha}(\tau)  N  d\tau||  + q)
\end{aligned}
\end{equation*}
where we use that $ \lim_{t \rightarrow \infty} || \int_0 ^t \tau^{\alpha-1}  E_{\alpha, \alpha}(A,\tau)  \nu (t-\tau) d\tau || = 0$ and $ \lim_{t \rightarrow \infty} E_{\alpha}(t)= 0$. Therefore,
\begin{equation*}
\limsup_{t \rightarrow \infty} ||x(t)|| \leq  \limsup_{t \rightarrow \infty} ||x(t)|| \frac {2\mu + \mu^2}{2\mu +1} < \limsup_{t \rightarrow \infty} ||x(t)||
\end{equation*}
which is a contradiction if $\limsup_{t \rightarrow \infty} ||x(t)|| \neq 0$. Hence, $\limsup_{t \rightarrow \infty} ||x(t)||=0$, whereby $\lim_{t \rightarrow \infty} ||x(t)|| = 0$.

The proof for $0< \alpha_i < 2$ is similar using the ideas of the proof of Theorem 2.4.

\end{proof}

\begin{remark}
(i) Theorem 2.7(i) is the extension of the BIBO stability result \cite[Theorem 1]{Malti} for non commensurate systems.

(ii) From Theorem 2.7(ii) one obtains the corresponding first Lyapunov theorem for mixed order systems, using similar arguments used in the proof of Theorem 2.4.

(iii) Actually, the statement of Theorem 2.7 can be improved. In Theorem 2.7(ii) $Q \neq 0$ holding condition of Theorem 2.7(iii). Condition (\ref{ww}) is improved by the condition $|\arg (\lambda_i)| > \frac{\alpha_i \pi}{2}$ for every $i=1, \ldots, n$.

(iv) Theorem 2.7(iii) holds using either that $\lim_{t \rightarrow \infty} Q(t)=0$ or that there exists $T\geq 0$ such that $\sup_{t\geq T} ||\tilde{Q}(t)|| < \frac{1}{C(\alpha, A)}$ where $C(\alpha, A)  := \sup_{t\geq 0}  \int_0 ^t  \tau^{\alpha-1} ||  E_{\alpha, \alpha}(\tau) || d\tau < \infty$ instead of $\sup_{t\geq 0} || \int_0 ^t  \tau^{\alpha-1}   E_{\alpha, \alpha}(\tau)\tilde{Q}(t-\tau) || d\tau$.
\end{remark}

\section{Robustness and convergence of fractional adaptive schemes}

\setcounter{section}{3}
\setcounter{equation}{0}\setcounter{theorem}{0}

In \cite{FA}, the importance in (control or identification) adaptive problems of the following system of equations, called Fractional Error Model of Type I, was stressed,
\begin{equation} \label{errorI}
\begin{cases} e = \phi^T w + \nu  \\ D^{\alpha} \phi = - \gamma e w = - \gamma w w^T \phi - \gamma w \nu  \end{cases}
\end{equation}
where $\alpha_i \in (0,2)$ for $i=1, \ldots, n$ and $D^{\alpha} \phi$ is the vector of components $D^{\alpha_i} \phi_i$ for $i=1, \ldots, n$, $e: [0, \infty) \rightarrow \mathbb{R}$ is a measure of the adaptation error, $w: [0, \infty) \rightarrow \mathbb{R}^n$ is the information signal, $\phi: [0, \infty) \rightarrow \mathbb{R}^n$ is the parameter error and $\nu: [0, \infty) \rightarrow \mathbb{R}$ is a perturbation function due to noise, imperfect modeling or initial conditions terms. $\gamma>0$ is an scalar or matrix parameter that can be suitably chosen to handle the speed of convergence. It is assumed that $w$ is bounded, otherwise it can be normalized using $\gamma$ (e.g. $\gamma= \frac{\gamma'}{1 + w^Tw}$). For simplicity, we assume $\gamma=1$ in the following. We will make abstraction of the specific problem where this scheme comes from.

We seek for conditions on $w$ so that $\phi$ converges to zero. Since $w$ is bounded, we will also have that the adaptation objective is reached, namely $\lim_{t \rightarrow \infty} e=0$. To use theorems stated in Section 2 we must take $w$ such that
\begin{equation} \label{connection}
-w(t) w^T(t) \leq A + Q(t)
\end{equation}
with $\sup_{t \geq T} ||Q||$ small enough for some $T\geq0$ and a suitable matrix $A$.

\subsection{Ideal case: $\nu=0$}

For those $w$ satisfying condition (\ref{connection}) in equality, that is, $-w(t) w^T(t) = A + Q(t)$ then $w$ is allowed to be just piece-wise continuous by Theorem 2.5 (cf. \cite[Theorem 2, Lemma 3]{FA}) where differentiability on $w$ was required). By applying Theorem 2.3 and Remark 2.4, we can extend the fractional adjustment to $\alpha \geq 1$ whereas in \cite{FA} it was proved only for $\alpha \in (0,1]$. By applying Theorem 2.7(iii), different orders on the adjustment parameters law is allowed (cf. \cite{FA} where the same order was used). Note that if $A$ is chosen with  negative real eigenvalues (e.g. $A= -\epsilon I$ for $\epsilon >0$), the condition on the eigenvalues of $A$ of above theorems holds. We remark that since the choice of the order of derivation on the adjustment parameters law is a relevant variable of optimization \cite[Example 5]{FA}, the established theorems effectively ample the tools for the designer of adaptive schemes.

For the scalar case, condition (\ref{connection}) takes the form $-w^2 \leq -a + q(t)$. If, for instance, $w^2 = \epsilon$ we choose $a=\epsilon$ and $q=0$. If  $\lim_{t \rightarrow \infty} w^2(t) = \epsilon$ we choose $a=\epsilon$ and $q=w^2 - \epsilon$ (it follows from Corollary 2.1). If $w^2 = sin^2 t$, we know that $sin^2 t > \epsilon$ up to a small periodicity. By taking a small enough $\epsilon$, we can define a function $q(t)$ zero everywhere up to a small periodicity. We note that those examples are special cases for which $w$ is persistently exciting. This can be formalized and generalized in the following way.

By definition, if $w:[0, \infty) \rightarrow  \mathbb{R}^n$ belongs to$   PE(n)$ then there exists $T_0>0$ and $\epsilon>0$ such that for any $t \geq 0$, $\int_t ^{t+ T_0} ww^T d\tau \geq \epsilon T_0 I$ where $I$ is the identity matrix (see e.g. \cite[\S 1]{Narendra}). Therefore, it is natural to write for any $t$, $w(t) w^T(t) \geq \epsilon I + Q(t)$ where $Q(t)$ vanishes in $[t,t+ T_0]$ or it is less or equal in norm than $\epsilon$ (e.g. $q(t) := sin^2(t) - \epsilon $ if $sin^2(t) < \epsilon$ else $0$, then $sin^2(t) \geq \epsilon + q(t)$. Note that $q$ is a piece continuous function). To use Theorem 2.1 (Remark 2) or Theorem 2.2, we require to choose an small enough $\epsilon>0$ such that $\epsilon C(\alpha, \epsilon I) <1$, which is justified since $E_{\alpha, \alpha}(t,A)$ is polynomial in $A$ (see e.g. \cite[Equation (1.56)]{Podlubny}) and $C(\alpha, \epsilon I)< \infty$. Thus, this choice of $Q$ guarantees for $w \in PE(n)$ the convergence of $\phi$, a hence of $e$, to zero.

On the other hand, one can use condition (\ref{Q}). We will first  prove that since $\sup_{t \geq 0} \int_0 ^t || \tau^{\alpha-1} E_{\alpha, \alpha}(\tau ^{\alpha} \epsilon I)|| d\tau =  C(\alpha, \epsilon I) < \infty$, then there exists a non negative almost periodic pulse function $p=p(t)$ (namely, there exists a time $T_1$ so that for any $t_0>0$, the function $p$ is not null in the interval $[t_0, t_0 + T_1]$) of amplitude $p_0 >0$ such that $\int_0 ^t || \tau^{\alpha-1} E_{\alpha, \alpha}(\tau ^{\alpha} \epsilon I)|| p(\tau) d\tau < 1$ for every $t$ and hence $\sup_{t \geq 0} \int_0 ^t || \tau^{\alpha-1} E_{\alpha, \alpha}(\tau ^{\alpha} \epsilon I)|| p(\tau) d\tau < 1$.

If $C(\alpha, \epsilon I) < 1$ the claims is obvious, by choosing $p_0 \leq 1$, since  we have $\int_0 ^t || \tau^{\alpha-1} E_{\alpha, \alpha}(\tau ^{\alpha} \epsilon I)|| p(t-\tau) d\tau \leq \int_0 ^t || \tau^{\alpha-1} E_{\alpha, \alpha}(\tau ^{\alpha} \epsilon I)|| d\tau < 1$ for any almost periodic pulse.

Now suppose $C(\alpha, \epsilon I) \geq 1$. One alternative is decreasing $\epsilon$ such that $C(\alpha, \epsilon I) < 1$. But this process could not be enough general. We instead reason by contradiction. If it were not the case that such pulse exists, one could take a sequence $(p_n)_{n \in \mathbb{N}}$ of almost periodic pulse functions $p_n$ of wide converging to zero such that $\int_0 ^t (t-\tau)^{\alpha-1} ||E_{\alpha, \alpha}((t-\tau) ^{\alpha} \epsilon I)|| \>  p_n(\tau) d\tau > 1$ and with the property of $\frac{1}{p_0}\sum_n \int_0 ^t (t-\tau)^{\alpha-1} ||E_{\alpha, \alpha}((t-\tau) ^{\alpha} \epsilon I)|| \>  p_n(\tau) d\tau = \int_0 ^t (t-\tau)^{\alpha-1} ||E_{\alpha, \alpha}((t-\tau) ^{\alpha} \epsilon I)|| d\tau < \infty$. For this, it is enough to define $p_n$ in an n-partition of the interval $[0, T_1]$ and recursively on $[T_1, T_2 + T_1]$. But this is a contradiction, since $\frac{1}{p_0} \sum_n \int_0 ^t (t-\tau)^{\alpha-1} ||E_{\alpha, \alpha}((t-\tau) ^{\alpha} \epsilon I)|| \>  p_n d\tau \geq  \frac{1}{p_0}\sum_n 1 \rightarrow \infty$ as $n \rightarrow \infty$.  Hence, we have a pulse function $p$ such that $\sup_{t \geq 0} \int_0 ^t || \tau^{\alpha-1} E_{\alpha, \alpha}(\tau ^{\alpha} \epsilon I)|| p(t-\tau) d\tau < 1$.

Note that any non negative almost periodic pulse $p'$ such that $p' \leq p$ share the same property. Therefore we can assume that $p_0\leq 1$. We repeat the process for function $f(t)= E_{\alpha, \alpha}(\tau ^{\alpha} \epsilon I) - E_{\alpha, \alpha}(\tau ^{\alpha} \epsilon I) p(\tau)$ instead of $E_{\alpha, \alpha}(\tau ^{\alpha} \epsilon I)$ (function $f$ being the same as $E_{\alpha, \alpha}(\tau ^{\alpha} \epsilon I)$ but cut off by $p$) until we get a function $f$ with $\sup_{t \geq 0} \int_0 ^t || \tau^{\alpha-1} f(\tau) )|| d\tau =  C_f(\alpha, \epsilon I) < 1$. In that case, we get that any almost periodic pulse with $p_0 \leq 1$ will give $\sup_{t \geq 0} \int_0 ^t || \tau^{\alpha-1} E_{\alpha, \alpha}(\tau ^{\alpha} \epsilon I)|| p(t-\tau) d\tau < 1$

Taking, $ ||Q(t)|| \leq p(t)$, we have $\int_0 ^t (t-\tau)^{\alpha-1} ||E_{\alpha, \alpha}((t-\tau) ^{\alpha} \epsilon I) Q(t)||  d\tau < 1$,  whereby $w w^T \geq \epsilon I + Q(t)$ and $Q$ holds condition (\ref{Q}). Hence, for $w \in PE(n)$ we have $\phi$ and $e= \phi^T w$ converging to zero.

\subsection{Perturbed case: $\nu \neq 0$}

We first show that there exist bounded destabilizing perturbations  justifying why it is worth to design robust schemes. Consider the scalar case and the following system
\begin{equation*}
\begin{cases}  e= \phi w + \nu \\ D^{\alpha} \phi= - e w \\ w = e \\ \nu= 1\end{cases}
\end{equation*}

Thus,
\begin{equation*}
D^{\alpha} \phi= - \frac{1}{(1-\phi)^2}.
\end{equation*}

Let us consider $\alpha=1$. Since $\frac{1}{(1-\phi)^2} > 0$, $\phi$ is decreasing. Hence, $\phi$  converges to some value or it diverges to $-\infty$. Suppose that $\lim_{t \rightarrow \infty} \phi =  L$ then $\lim_{t \rightarrow \infty} D \phi = - \frac{1}{(1-L)^2}  < 0$, which is a contradiction since $\phi$ converges and in particular it is bounded. Then $\phi$ diverges to $-\infty$ and $\lim_{t \rightarrow \infty} (e= \frac{\nu}{1-\phi})=0$.

Let us consider consider $\alpha >0$. By a similar reasoning there are two possibilities either $\phi$ is bounded or diverges to $-\infty$. However if $a <\phi(t) < b$ for some $a,b \in \mathbb{R}$ then $ c < D^{\alpha} \phi < d < 0$ for negative numbers $c,d \in \mathbb{R}$. By $\alpha$-integrating we have the same contradiction as above. Then $\phi$ diverges to $-\infty$ and $\lim_{t \rightarrow \infty} e= 0$.

Therefore, $\nu \equiv 1$ is a destabilizing perturbation for such a choice of $w$. Note that $w \notin PE(1)$ since $\lim_{t \rightarrow \infty}w (t)= 0$.  A similar construction was done in \cite[\S III.2]{Narendra} for integer Error Model of Type II for the specific control problem, where the assumption $e=w$ comes out naturally for an specific control problem.

Let us consider now $\nu \neq 0$ any bounded continuous function. System (\ref{errorI}) can be expressed in the following form
\begin{equation*}
D^{\alpha}\phi(t)= -w(t)w^T(t)\phi(t) - \nu(t) w(t).
\end{equation*}

When $\alpha=1$, $w \in PE(n)$ and $\nu \equiv 0$, it was proved in \cite{Mor} that $\phi$ converges exponentially to zero. That implies that the state transition matrix $\Phi(t,t_0)$ converges exponentially to zero . Since the solution for a general $\nu$ can be written as
\begin{equation}\label{tv}
\phi(t)= \Phi(t,t_0)\phi(t_0) - \int_0^t \Phi(t-\tau,t_0)  \nu(\tau) w(\tau) d\tau
\end{equation}
where $\Phi(t,t_0)$ converges exponentially to zero (recall that $\phi(t)=\Phi(t,t_0)\phi(t_0)$) and $\nu w$ is bounded, the integral term is bounded, whereby $\phi$ is bounded.

A direct generalization of the above argument to the case $\alpha >0$  has the following difficulties. In principle, one can postulate for the system
\begin{equation*}
D^{\alpha}x(t)= A(t)x(t)
\end{equation*}
a matrix function $\Phi(t)$ so that $x(t)=\Phi(t)x(0)$ for any $x(0) \in \mathbb{R}^n$. Note that since Caputo derivative is initialized at $t=0$, there is no need to use $\Phi(t,t_0)$ since $t_0=0$ is fixed. Therefore, by $\alpha-$differentianting $x(t)=\Phi(t)x(0)$ and given that $x(0)$ is arbitrary, function $\Phi(t)$ must satisfies the following equation
\begin{equation*}
D^{\alpha}\Phi(t)= A(t)\Phi(t)
\end{equation*}
with $\Phi(0)=I$ the identity matrix. With this matrix one postulates as solution to the system
\begin{equation*}
D^{\alpha}x(t)= A(t)x(t) + f(x)
\end{equation*}
the expression
\begin{equation}\label{var}
x(t)= \Phi(t)x(0) + I^{\alpha} [\Phi(t-\cdot) f(\cdot)] (t).
\end{equation}

Taking Caputo $\alpha$-derivative of (\ref{var}) and using that $D^{\alpha}I^{\alpha}[f(\cdot)](t)=f(t)$ for $f$ a differentiable function, we obtain $D^{\alpha}x(t)= A(t)x(t) + f(x)$. By uniqueness of the solution, we conclude that (\ref{var}) is the solution. However, neither in the series of papers of Cong et al. nor other revised literature a condition is stated to estimate the speed of convergence or more precisely conditions to have $\Phi(t) \in \mathcal{L}^1$ for time varying linear systems (see also Remark 1.(ii). In the proof of \cite[Theorem 2]{FS} an estimate for the speed of convergence of $x$ can be deduced but it is not enough to conclude that  $\Phi(t) \in \mathcal{L}^1$. Moreover, after \cite[Lemma 3.1]{Cong.L} non negative exponential order is to be hoped in linear time varying Caputo fractional systems.

Consider $w \in EP(n)$ such that  $-w(t) w^T(t) = A + Q(t)$ where $Q(t)$ and matrix $A$ are chosen as in $\S 3.1$. For any additive bounded perturbation in the error measurement, we have
\begin{equation*}
D^{\alpha}\phi(t)= (A + Q(t))\phi(t) - \nu(t) w(t).
\end{equation*}

Hence, Theorem 2.1 can be applied since $\nu w$ is a bounded function, which gives a robustness result for adaptive schemes without modifying the adaptive laws, guaranteeing that for vanishing disturbances its properties remains unchanged. Moreover, by Theorem 2.1(ii), if $\nu$ converges to zero, then $\phi$ converges to zero. From Theorem 2.7,
the results hold for different order of derivation in the adjustment of each parameter.


\subsection{Error Model of Type II}

We study the Error Model of Type II defined by the following system
\begin{equation} \label{errorII}
\begin{cases} D^{\alpha} e = Ae + \phi^T w + \nu  \\ D^{\beta} \phi = -  e w \end{cases}
\end{equation}
where $e: [0, \infty) \rightarrow \mathbb{R}^n$,  $\phi: [0, \infty) \rightarrow \mathbb{R}^n$, $w: [0, \infty) \rightarrow \mathbb{R}^n$ and   $\alpha, \beta$ are seen as vectors of non negative components, in such a way that e.g. $ D^{\alpha} e $ is a vector of components $D^{\alpha_i} e_i $ for $i=1. \ldots, n$.

This system was studied in \cite{FA}, for the special case of $\alpha=\beta \leq 1$  but none explicit condition was established for the convergence to zero of the error.

Defining $x= (e^T, \phi^T)^T$, system (\ref{errorII}) is equivalent to
\begin{equation*}
D^{\alpha} x = \begin{bmatrix}{A}&{w^T}\\{-w}&{0}\end{bmatrix} x  +  \begin{bmatrix}{\nu}\\{0}\end{bmatrix}.
\end{equation*}

Again, to apply theorems of $\S 2$ we must express
\begin{equation} \label{wII}
\begin{bmatrix}{A}&{w^T(t)}\\{-w(t)}&{0}\end{bmatrix} = \Lambda + Q(t)
\end{equation}
for suitable $\Lambda$ and $Q$. This condition is in fact a condition on $w$. For such $w$, Theorem 2.7 guarantees convergence to zero of $(e,\phi)$ if $\nu \equiv 0$ and boundedness of  $(e,\phi)$ if $\nu$ is bounded.  Theorem 2.2 allows to relax condition (\ref{ww}) to an inequality when $\alpha_i=\alpha<1$ and Theorem 2.5 to piece continuous functions $w$. When $\alpha=1$, $\nu \equiv 0$ and $w$ is such that for all $t>0$ and for all constant unit vector $u \in \mathbb{R}^n$ there exists $T_0, \epsilon$ with $|\int_t^{t+T_0} u^Tw(\tau) d\tau| \geq \epsilon T_0$, the system is exponentially stable, whereby if $\nu$ is bounded $x$ is bounded \cite{Narendra}. So, it seems natural to look into the set of functions that hold condition (\ref{ww}).

We examine as a simple case $w, e$ scalars, $w=w_0 \neq 0$ constant, $A$ is a constant $-a < 0$. Taking $\Lambda := \begin{bmatrix}{-a}&{w_0}\\{-w_0}&{0}\end{bmatrix}$, its characteristic polynomial is $\lambda^2 + a \lambda + w_0^2=0$. Then, $\Lambda$ is stable in the integer sense and condition (\ref{ww}) trivially holds  (eventually, in contrast with the case $\alpha=1$, $a$ could be negative to have $\Lambda$ stable in the fractional sense).

Now if $w= w_0 + q(t)$ where $q$ converges to zero or it is small enough, then, $\Lambda  + Q := \begin{bmatrix}{-a}&{w_0}\\{-w_0}&{0}\end{bmatrix}+ \begin{bmatrix}{0}&{q}\\{q}&{0}\end{bmatrix}$. Therefore $ || Q(t) ||_F = |q(t)| || \begin{bmatrix}{0}&{1}\\{1}&{0}\end{bmatrix} ||_F  = 2 |q(t)|$ and condition $\sup_{t\geq T} ||Q(t)|| < \frac{1}{C(\alpha, A)}$ of Corollary 2.1 holds for $||q||_\infty$ small enough.

More generally, note that the condition $|\int_t^{t+T_0} u^Tw(\tau) d\tau| \geq \epsilon T_0$ implies that we can write $w(t)= \pm \epsilon u + q(t)$ where $q(t)$ is such that $\int_t^{t+T_0} u^Tq(t) d\tau \geq 0$. Then, we can write $\Lambda + Q(t) = \begin{bmatrix}{-a}&{\pm \epsilon u^T}\\{\mp \epsilon u}&{-\mu I}\end{bmatrix} + \begin{bmatrix}{0}&{q^T}\\{-q}&{\mu I}\end{bmatrix}$, where $\mu>0$ is chosen small enough. The conditions for $Q$ are similar as in $\S 3.1$.

\section*{Acknowledgements}

 The authors thank CONICYT-Chile for the support under Grant No FB0809 'Centro de Tecnologia para la Mineria' and FONDECYT 1150488 'Fractional Error Models in Adaptive Control and Applications'.




 \bigskip \smallskip

 \it

 \noindent
$^1$ Department of Electrical Engineering \\
University of Chile \\
Av. Tupper 2007 \\
Santiago, CHILE  \\[4pt]
e-mail: jgallego@ing.uchile.cl 
\hfill Received: September 15, 2016 \\[12pt]
$^2$ Department of Electrical Engineering \\
University of Chile \\
Av. Tupper 2007 \\
Santiago, CHILE \\ [4pt]
e-mail: mduarte@ing.uchile.cl
\end{document}